\documentclass[11pt,reqno]{article}   %style-phys/cs 
\usepackage{fullpage}

\usepackage{csquotes}
\usepackage{mathrsfs}
\usepackage{yfonts}
\usepackage[T1]{fontenc}
\usepackage[utf8]{inputenc}
\usepackage{lmodern}
\usepackage{textcomp}
\usepackage{setspace}
\usepackage{tgothic}

\usepackage{etoolbox}
\usepackage{xparse}

\usepackage{mathtools}
\usepackage{amsthm,thmtools}
\usepackage{dsfont}
\let\mathbb\mathds
\usepackage{amssymb}
\usepackage[scr]{rsfso} % rather than mathrsfs
\usepackage{bm}
\usepackage{euscript}
\usepackage{amsbsy}
\DeclareMathAlphabet\mathbfcal{OMS}{cmsy}{b}{n}

\usepackage{microtype}
\usepackage{authblk}
\usepackage{graphicx} % lo de demo es para que no de errores al compilar cuando no están las imágenes cerca
\usepackage{tikz}
\usetikzlibrary{cd,quotes}
\usetikzlibrary{decorations.markings}
\usetikzlibrary{decorations.pathreplacing}
\usepackage{pgfplots}
\pgfplotsset{compat=1.13}
\usepackage{booktabs}
\usepackage{float}
\usepackage{environ}

\usepackage{xcolor} 

\usepackage{soul}
\usepackage[shortlabels]{enumitem}
\usepackage{eso-pic}

\usepackage[normalem]{ulem}

\definecolor{bluegray}{rgb}{0.4, 0.6, 0.8}
\definecolor{turquoise}{rgb}{0.2, 0.7, 0.6}
\usepackage{hyperref} % éste va siempre el último (por lo general)

\usepackage{tikz-cd}
\tikzset{Rightarrow/.style={double equal sign distance,>={Implies},->},
	triple/.style={-,preaction={draw,Rightarrow}},
	quadruple/.style={preaction={draw,Rightarrow,shorten >=0pt},shorten >=1pt,-,double,double
		distance=0.2pt}}

\def\scr{\EuScript}
\def\on{\operatorname}
\def\sf{\mathsf}

\def\CC{\mathbb{C}}
\def\ZZ{\mathbb{Z}}
\def\hH{\scr{H}}

\def\Fun{\on{Fun}}

\def\Hom{\on{Hom}}

\def\id{\on{id}}
\def\Set{{\sf{Set}}}

\def\RR{\mathbb{R}}
\def\St{\on{St}}
\def\HC{\on{HC}}
\def\SimpEff{{\sf{SimpEff}}}
\def\Eff{{\sf{Eff}}}
\def\one{\mathbb{1}}
\def\zero{\mathbb{0}}
\def\Den{\on{Den}}
\def\Herm{\on{Herm}}
\def\Proj{\on{Proj}}
\def\Tr{\on{Tr}}

\newcommand{\column}[3]{\ensuremath{\left( \begin{matrix}
#1 \\
#2 \\
#3
\end{matrix}  \right) }}

\DeclarePairedDelimiterX\set[1]{\lbrace}{\rbrace}{#1}

\declaretheoremstyle[bodyfont=\itshape,notefont=\bfseries]{abellanA}
\declaretheoremstyle[notefont=\bfseries]{abellanB}

\declaretheorem[style=abellanA,numberwithin=section,name={Theorem}]{theorem}

\declaretheorem[style=abellanA,numberlike=theorem,name={Lemma}]{lemma}
\declaretheorem[style=abellanB,numberlike=theorem,name={Definition}]{definition}
\declaretheorem[style=abellanB,numberlike=theorem,name={Remark}]{remark}
\declaretheorem[style=abellanB,numberlike=theorem,name={Construction}]{construction}
\declaretheorem[style=abellanA,numberlike=theorem,name={Proposition}]{proposition}
\declaretheorem[style=abellanB,numberlike=theorem,name={Example}]{example}

\declaretheorem[style=abellanA,numberlike=theorem,name={Corollary}]{corollary}

\docsvlist{theorem,lemma,definition,remark,proposition,example,corollary}

\let\leq\leqslant
\let\geq\geqslant
\let\epsilon\varepsilon

\usepackage[bbgreekl]{mathbbol}

\DeclareMathOperator*\colim{colim}
\def\op{{\on{op}}}

\newlength\squareheight
\newcommand\squareslashleft{\tikz{\draw (0,0) rectangle (\squareheight,\squareheight);\draw(0,0) -- (\squareheight,\squareheight)}}
\DeclareMathOperator\squarediv{\squareslashleft}
\newcommand\squareslashright{\tikz{\draw (0,0) rectangle (\squareheight,\squareheight);\draw(\squareheight,0) -- (0,\squareheight)}}
\DeclareMathOperator\squaredivright{\squareslashright}

\title{{Simplicial effects and weakly associative partial groups}}
 
	\author[1]{Cihan Okay\footnote{cihan.okay@bilkent.edu.tr}}
	
	\author[1]{Victor Castillo\footnote{victor.torres@cimat.mx}}
	
	\author[2]{Walker H. Stern\footnote{walker.stern@tum.de}}

	\affil[1]{Department of Mathematics, Bilkent University, Ankara, Turkey}
	\affil[2]{Department of Mathematics, Technical University of Munich, Munich, Germany}

\date{\today}

\begin{document}
    \maketitle

\begin{abstract} 
In this paper, we introduce a new category of simplicial effects that extends the categories of effect algebras and their multi-object counterpart, effect algebroids. Our approach is based on relaxing the associativity condition satisfied by effect algebras and, more generally, partial monoids. Within this framework, simplicial effects and weakly associative partial groups arise as two extreme cases in the category of weak partial monoids. Our motivation is to capture simplicial structures from the theory of simplicial distributions and measurements that behave like effects.
\end{abstract}

	\tableofcontents

\section{{Introduction}}

Quantum measurements can be described using the theory of effect algebras, introduced in \cite{foulis1994effect}. A canonical example is the effect algebra of projectors on a Hilbert space, which forms the basis of projective measurement in quantum theory. When such a measurement is performed on a quantum state, it produces a probability distribution on the set of possible outcomes via the Born rule. In this sense, projectors act as ``effects'' in quantum theory.
Their salient algebraic properties are captured by effect algebras, which provide a partial operation satisfying an associativity condition. More precisely, these objects can be viewed as partial monoids in the sense of \cite{Segal}.

The simplicial formulation of measurements in quantum theory, introduced in \cite{sdist}, employs effect algebras to construct simplicial objects that behave like ``simplicial effects.'' More precisely, both the measurements and outcomes in this theory have a spatial structure specified by simplicial sets.
A simplicial set consists of a sequence of sets $X=\set{X_n}_{\geq 0}$, where each $X_n$ represents the set of $n$-simplices (basic building blocks of the structure),
together with simplicial structure maps, which specify how these simplices can be glued and collapsed. 
A \emph{simplicial distribution} on a measurement space $X$ and an outcome space $Y$ is then a simplicial set map 
\[
\begin{tikzcd}
p : X \arrow[r] & D(Y)
\end{tikzcd}
\]
where $D(Y)$ is the simplicial set whose $n$-simplices are finite-support probability distributions on $Y_n$.
The effect algebra of projectors can be used to define a simplicial set $P_\hH(Y)$ whose $n$-simplices are given by
\[
P_{\scr{H}}(Y)_n := \{\,\Pi : Y_n \to \on{Proj}(\scr{H})
\mid \Pi\text{ fin.\ supp.}, \,\sum_{y \in Y_n} \Pi(y) = \mathbb{1}\}.
\]
Given a \emph{simplicial measurement}, {that is}, a simplicial map $\Pi:X\to P_\hH(Y)$,  a density operator (i.e.\ a positive trace-$1$ operator) $\rho$ can be used to obtain a simplicial distribution of the form
\[
\begin{tikzcd}
p_\rho : X \arrow[r,"\Pi"] & P_\hH(Y) \arrow[r,"\rho_*"] & D(Y)
\end{tikzcd}
\]
where $\rho_*$ is defined by sending a projective measurement $\{\Pi(y)\}_{y\in Y_n}$ to the distribution $p_\rho(y) = \Tr\bigl(\rho\,\Pi(y)\bigr)$.
A key example arises when $Y$ is the nerve space $N(\ZZ_{/d})$ of the additive group $\ZZ_{/d}$. Then $P_\hH(N\ZZ_{/d})$ can be identified with the simplicial set $N(\ZZ_{/d}, U(\hH))$, the \emph{commutative $d$-torsion classifying space} of the unitary group $U(\hH)$. This space is a simplicial subset of the nerve $N(U(\hH))$, further constrained by commutativity and $d$-torsion conditions. Concretely, an $n$-simplex consists of $n$-tuples $(g_1,g_2, \dots, g_n)$ of pairwise commuting unitary operators satisfying $(g_i)^d = \one$ for all $1 \leq i \le n$.
It turns out that $N(\ZZ_{/d}, U(\hH))$ behaves as a partial monoid but with a weakened associativity property. The underlying issue is that a tuple $(a,b,c)\in \ZZ_{/d}^{\times 3}$ satisfying $(ab)c=(ab)c$ and $a(bc)=(bc)a$ does not necessarily consist of pairwise commuting elements. Hence, the commutative nerve carries a partial operation that associates more weakly than the usual partial monoids defined in \cite{Segal}.

The aim of this work is to introduce a generalization of effect algebras and effect algebroids \cite{Roumen} that accommodates algebraic structures associating more weakly than those in \cite{Segal}, yet more strictly than the partial groups in \cite{Chermak}. We call the resulting objects \emph{simplicial effects}.
To arrive at our definition, we first develop the notion of weak associativity. To this end, we introduce the following categories:
\begin{itemize}
\item The category $\sf{PM}$ of partial monoids, whose objects are partial monoids in the sense of \cite{Segal}.
\item The weakened version, the category $\sf{WPM}$ of weak partial monoids.
\item The category $\sf{Mag}$ of partial unital magmas where no associativity condition is imposed on the partial operation.
\item Finally, $\sf{Mag}^{\on{ad}}$, consisting of partial unital magmas equipped with associativity data capturing the notion of associativity introduced in \cite{Chermak}.
\end{itemize}
In the diagram below, we show how these categories relate to each other and to the category of simplicial sets:
\[
\begin{tikzpicture}
			\path (0,0) node (PM) {$\sf{PM}$} ;
			\path (3,0) node (WPM) {$\sf{WPM}$};
			\path (6,0) node (Mag) {$\sf{Mag}$};
			\path (9,0) node (Magad) {$\sf{Mag}^{\on{ad}}$};
			\path (0,-3) node (S1) {$\Set_\Delta$} ;
			\path (3,-3) node (S2) {$\Set_\Delta$};
			\path (6,-3) node (S3) {$\Set_\Delta$};
			\path (9,-3) node (S4) {$\Set_\Delta$};
			\draw[->] (PM) to node[above] {$\Psi$} (WPM);
			\draw[->] (WPM) to node[above] {$\Phi$} (Mag);
			\draw[->] (Mag) to node[above] {$(-)^\sharp$} (Magad);
			\draw[->] (PM) to (S1);
			\draw[->] (WPM) to (S2);
			\draw[->] (Mag) to (S3);
			\draw[->] (Magad) to node[right] {$N$} (S4);
			\draw[double] (S1) to (S2);
			\draw[double] (S2) to (S3);
			\draw[double] (S3) to (S4);

			\draw[blue,decoration={brace}, decorate] (9+.5,-3.5) node {} -- (0-.5,-3.5);
			\path[blue] (4.5,-3.9) node {spiny, reduced};
			\draw[blue,decoration={brace}, decorate] (6+.5,-4.5) node {} -- (0-.5,-4.5);
			\path[blue] (3,-4.9) node {2-coskeletal};
			\draw[blue,decoration={brace}, decorate] (3+.5,-5.5) node {} -- (0-.5,-5.5);
			\path[blue] (1.5,-5.9) node {weakly 2-Segal};
			%\draw[blue] (0,-3) ellipse (0.5 and 0.3);
			\draw[blue,decoration={brace}, decorate] (1,-6.5) node {} -- (0-.5,-6.5);
\path[blue] (0.25,-6.9) node {2-Segal};			
			
			%\path[blue] (-1.5,-3) node {2-Segal}; 
		\end{tikzpicture}
\]
As indicated in the top part of the diagram, there are fully faithful functors interrelating these categories. In addition, we define a nerve functor whose essential image in $\sf{Set}_\Delta$ is indicated in the diagram. The adjectives in blue describe the essential images of the vertical functors, and the bottom horizontal functors are identities. Together, these provide a complete characterization of the hierarchy of fundamental algebraic structures considered in this paper and their corresponding geometric and combinatorial interpretations.

Having established the notion of weak associativity, we proceed to examine invertibility. As it turns out, this concept is intrinsically linked to our desired objects of simplicial effects. We introduce the category of weakly associative partial groups by including the invertibility condition in weak partial monoids. Our motivating examples are the commutative $d$-torsion nerves. Since these objects represent one extreme case where every element has an inverse, we find that simplicial effects correspond to the other extreme, in which no non-trivial element has an inverse. To reach this conclusion, we reformulate Roumen's effect algebroids. In particular, we focus on the two key components in this definition: the \emph{orthocomplement} and the \emph{zero-in-one law}.
This brings us to the desired category of simplicial effects, denoted by $\sf{SimpEff}$, whose objects are spiny, inverseless, weakly $2$-Segal cyclic sets. We have fully faithful embeddings of categories
\[
\begin{tikzcd}
\sf{Eff} \arrow[r] & \sf{EffAlgd} \arrow[r] & \sf{SimpEff}. 
\end{tikzcd}
\]
We conclude by constructing a simplicial effect that is not an effect algebroid. Specifically, we modify $N(\ZZ_3,\CC^9)$ in such a way that the weak associativity property is preserved and the resulting partial monoid has no inverses. We demonstrate that this object is not an effect algebroid by showing that the associativity condition, formulated via the $2$-Segal condition, fails to hold.
Our construction also accounts for non-trivial states. We prove a version of Gleason’s theorem (for measurements with three outcomes) for this simplicial effect, demonstrating that its state space corresponds exactly to the space of all quantum states.

{
Our category of simplicial effects broadens the literature on effect algebras \cite{foulis1994effect,Roumen,RouCoh,cho2015introduction} and their connection to contextuality \cite{staton2018effect}. 
Recently, simplicial techniques have proven to be very useful in characterizing non-contextual distributions and extremal contextual distributions, leading to new results \cite{kharoof2023topological,kharoof2024extremal}. Similarly, we expect that a simplicial theory of effects will deepen our understanding of the fundamentals of quantum measurements. 
}

The structure of our paper is as follows: In Section \ref{sec:Motivation and preliminaries}, we introduce the motivating examples of effect algebras and effect algebroids from quantum foundations, along with the more recently defined simplicial distributions and measurements, and related nerve spaces such as the commutative nerve. Section \ref{sec:weak associativity} discusses three different versions of associativity of varying strength, where we also define the key category of weak partial monoids.
In Section \ref{sec:Nerves and associativity conditions}, we introduce a nerve construction and describe the essential images of our categories of interest, as depicted in the diagram above. Section \ref{sec:invertibility} focuses on invertibility considerations, culminating in the category of weakly associative partial groups. We show that commutative nerves are weakly associative partial groups. Finally, Section \ref{sec:Simplicial effects} presents the category of simplicial effects and gives a noteworthy example that is an object of this category but not an effect algebroid. This illustrates how our notion of simplicial effects captures new non-trivial simplicial structures exhibiting behavior analogous to effects.

\paragraph{Acknowledgments.}
This work is supported by the Air Force Office of Scientific Research (AFOSR) under award number FA9550-21-1-0002. The first author also acknowledges support from the Digital Horizon Europe project FoQaCiA, GA no. 101070558 and AFOSR FA9550-24-1-0257. 

\section{Motivation and preliminaries}
\label{sec:Motivation and preliminaries}

{Effect algebras and their generalization, effect algebroids, provide a formal way to capture quantum-theoretic notions of effects. Their principal role is to encode probability distributions stemming from quantum measurements. In the recent theory of simplicial distributions and measurements \cite{sdist}, new variants have emerged that exhibit similar behavior. In this section, we present our motivating constructions, which lead us toward the development of a unified theory of simplicial effects.}

%we collect background results and definitions necessary for the remainder of this paper. Our motivation throughout the paper is to generalize the notion of an effect algebra to settings which are not strongly associative. We begin by recalling this motivation, and explaining where associativity enters the picture. 

\subsection{Effect algebras}\label{subsec:effectalg}

As effect algebras are, in particular, partial monoids, we 
%begin by recalling 
{first recall}
the definition of a partial monoid.

\begin{definition}[\cite{Segal}]  \label{def:partial monoid}
{A {\it partially defined binary operation} on a set $M$ is a function $\cdot:M_2\to M$ from a subset $M_2\subset M\times M$. We say $a\cdot b$ is defined whenever $(a,b)\in M_2$.}
	 A \emph{partial monoid} $(M,\cdot,1)$ is a set $M$ equipped with a partially defined 
	 %map $\cdot$ from $M\times M$ to $M$ 
	 {binary operation on $M$}
	 such that 
	\begin{itemize}
		\item For $a,b,c\in M$, $a\cdot (b\cdot c)$ is defined if and only if $(a\cdot b)\cdot c$ is defined, and if both are defined they are equal.
		\item For every $a\in M$ both $1\cdot a$ and $a\cdot 1$ are defined and equal to $a$. 
	\end{itemize}
	A \emph{morphism of partial monoids} is a map $f:M\to N$ of sets such that if $a\cdot b$ is defined in {$M$}, so is $f(a)\cdot f(b)$ and $f(a\cdot b)=f(a)\cdot f(b)$. 
\end{definition}

%{We will return to the definition of partial monoids more systematically later, however, the above is sufficient to introduce our main definition. }
%{Given a partial monoid $M$, we will write $M_2\subset M^{\times 2}$ for the domain of the partial operation $\cdot$.}

\begin{definition}[\cite{foulis1994effect}]  \label{def:effect algebra}
	An \emph{effect algebra} $(E,+,0,\perp)$ is a partial monoid equipped with a map 
$(-)^\perp:E\to E$ called the \emph{orthocomplement}. 
%From these data, 
{W}e define 
	%$1:=(0)^\perp$. 
	{$1:=0^\perp$.}	
	These data must satisfy:
	\begin{enumerate}
		\item (Commutativity) If $(a,b)\in E_2$ then $(b,a)\in E_2$, and $a+b=b+a$.
		\item (Orthocomplement) For $a\in E$, $a^\perp$ is the unique element such that $(a,a^\perp)\in E_2$ and $a+a^\perp=1$. 
		\item (Zero-in-one) If $(a,1)\in E_2$, then $a=0$.
	\end{enumerate}   
	 A \emph{morphism of effect algebras} is a morphism $f:E\to F$ of partial monoids  which preserves the orthocomplement in the sense that $f(a^\perp)=f(a)^\perp$ for all $a\in E$. We will write $\Eff$ for the category of effect algebras.		
\end{definition}

\begin{example}\label{ex:example effect algebras}
{We now illustrate the notion of effect algebras with the following canonical examples.}
	\begin{enumerate}
		\item The interval $[0,1]\subset \mathbb{R}$ is an effect algebra when equipped with the partial operation $+$, the unit $0$, and the orthocomplement $a\mapsto 1-a$. 
		
		\item Any Boolean algebra $B$ is an effect algebra, with partial operation given by the sum of disjoint elements. The orthocomplement is precisely the complement. 
		
		\item {Let $\hH$ be a finite-dimensional Hilbert space. The prototypical example of effect algebras is the set of projection operators, denoted by $\Proj(\hH)$.  The orthocomplement is given by orthogonal complementation. The partial operation is given by addition when two projectors are orthogonal with the zero element given by the zero operator $\zero$. We have $\zero^\perp  =\one$, the identity operator.}
	\end{enumerate}	
	
\end{example}

\begin{definition}\label{def:state}
	An effect algebra morphism $\varphi:E\to [0,1]$ is called a \emph{state {on $E$}}. We will write $\St(E)$ for the set of states {on $E$}.
\end{definition}

A foundational result that characterizes states on the effect algebra of projectors is the Gleason's theorem \cite{gleason1975measures}:

\begin{theorem}\label{thm:Gleason}
	Let $\hH$ be a finite-dimensional Hilbert space of dimension $\geq 3$ and $\Den(\hH)$ denote the set of density operators (quantum states), i.e., trace $1$ positive operators. Then
	\[
	\begin{tikzcd}[row sep=0em]
		\Den(\hH) \arrow[r] & \St(\Proj(\hH)) \\
		\rho \arrow[r,mapsto] & (\Pi \mapsto \Tr(\rho\Pi))
	\end{tikzcd} 
	\]
	is a bijection.
\end{theorem}

\subsection{Simplicial sets}

One of our main aims is to fit effect algebras into the setting of \cite{sdist}, in which quantum measurements and distributions are studied using simplicial sets.  
We here collect some of the background on simplicial sets which we will use in the sequel.

\begin{definition}
	The \emph{simplex category} $\Delta$ is the %skeletal 
	category whose objects are the standard finite non-empty ordinals 
	\[
	[n]:=\{0<1<\cdots<n\}
	\] 
	and whose morphisms are the weakly monotone maps of sets{, i.e., $f(i)< f(j)$ whenever $i< j$} . A \emph{simplicial set} is a functor $X:\Delta^\op \to \Set$, and the \emph{category of simplicial sets} is the functor category $\Set_\Delta:=\Fun(\Delta^\op,\Set)$. 
\end{definition}

	A simplicial set can be equivalently considered as a collection $\{X_n\}_{n\geq 0}$ of sets together with \emph{face maps} $d_i:X_n\to X_{n-1}$ and \emph{degeneracy maps} $s_i:X_{n}\to X_{n+1}$ satisfying certain relations called the \emph{simplicial identities} \cite[Ch 1, Eq. 1.3]{GoerssJardine}. These maps correspond, respectively, to the injective map $d^i:[n-1]\to[n]$ in $\Delta$ which skips $i$, and the surjective map $s^i:[n+1]\to [n]$ in $\Delta$ which doubles {the preimage of} $i$. The set $X_n$ is called the \emph{set of $n$-simplices} of $X$.
%\end{remark}

\begin{example}
	For every category $\scr{C}$, there is a simplicial set $N(\scr{C})$ called the nerve of $\scr{C}$ whose set of $n$-simplices $N(\scr{C})_n$ is the set of composable   $n$-tuples $(f_1,\ldots,f_n)$ of morphisms of $\scr{C}$. The face and degeneracy maps are given, respectively, by 
	\[
	d_i(f_1,\ldots,f_n)=\begin{cases}
		(f_2,\ldots,f_n) & i=0\\
		(f_1,\ldots, f_{n-1}) & i=n \\
		(f_1,\ldots, f_{i+1}\circ f_i,\ldots, f_n) & \text{else}
	\end{cases}
	\]
	and 
	\[
	s_i(f_1,\ldots,f_n)=(f_1,\ldots,\underbrace{\on{id}}_{i^{\on{th}}},\ldots, f_n)
	\]
	The functor which sends $\scr{C}$ to $N(\scr{C})$ is fully faithful, and so we can consider categories as simplicial sets. 
	
	Since monoids (and thus groups) can 
	%each 
	be considered as categories with one object, we can take the nerve of a monoid as well which we will 
%abusively 
denote by $N(M)$.  
\end{example}

\begin{example}
	We denote the representable simplicial set $\Delta(-,[n])$ by $\Delta^n$, and call it the \emph{standard (combinatorial) $n$-simplex.}   
We denote by $\partial \Delta^n\subset \Delta^n$ {the \emph{boundary of the $n$-simplex},} the simplicial subset consisting of all simplices which factor through a proper subset of $[n]$.
\end{example}

%\begin{definition}
%	We will make use of several subcategories of $\Delta$, namely $\Delta_{\on{inj}}\subset \Delta$ which is the wide subcategory which includes precisely the injective maps, and the subcategory $\Delta_{\leq k}\subset \Delta$, which is the full subcategory on the objects $[n]$ such that $n\leq k$. 
%\end{definition}

{Let} $\Delta_{\leq k}\subset \Delta$ denote the full subcategory on the objects $[n]$ such that $n\leq k$. 
Composition with the inclusion $\iota_k:\Delta_{\leq k}^\op \to \Delta^\op$ induces a functor 
\[
\begin{tikzcd}
	\on{tr}_k:&[-3em]\Set_\Delta\arrow[r] & \Set_{\Delta_{\leq k}}
\end{tikzcd}
\]
called the \emph{$k$-truncation functor}. By general abstract nonsense, this admits a right adjoint called \emph{$k$\textsuperscript{th} coskeleton functor} and written $\on{cosk}_k$. We call a simplicial set in the image of this functor \emph{$k$-coskeletal}. 
%Our final preliminary collects some basic results on coskeletal simplicial sets. 

\begin{proposition}\
	Let $X\in \Set_\Delta$ be a simplicial set. 
	\begin{enumerate}
		\item If $X$ is isomorphic to the nerve of a category, then $X$ is 2-coskeletal. 
		\item The simplicial set $X$ is $k$-coskeletal if and only if, for every $n>k$, any extension problem 
		\[
		\begin{tikzcd}[column sep=huge,row sep=large]
			\partial \Delta^n\arrow[r]\arrow[d,hookrightarrow] & X \\
			\Delta^n \arrow[ru,dashed]
		\end{tikzcd}
		\]
		has a \emph{unique} solution, {i.e., there exists $\Delta^n\to X$ making the diagram commute.}
	\end{enumerate}
\end{proposition}
\begin{proof}
	Both statements are well-known; {see \cite{Duskin} and \cite[{Tag 051Z}]{kerodon}.} %To our knowledge, the first was first proven by Duskin in \cite{Duskin}, and it is unclear to us when the second first appeared.  For a more recent overview, see Kerodon, \cite[{Tag 051Z}]{kerodon} \comm{should we keep this ref}.
\end{proof}

\subsection{Commutative nerves}

	Our  main
	 class of motivating examples for weakly associative partial groups, to be defined in Section \ref{sec:invertibility}, come directly from groups. 
	 These are the commutative nerves and $d$-torsion commutative nerves of groups introduced in \cite{AdemCohenTorres,okay2021commutative}.  
%	 Informally speaking, the commutative nerve of a group $G$ is the weak partial monoid whose underlying set is $G$, but in which we only consider elements to be multiplicable if they commute. However, as the name implies, it is more convenient for us to define commutative nerves of groups as simplicial sets, as follows. 

	\begin{definition}
		Let $G$ be a group. Define the \emph{commutative nerve} of $G$ to be the {simplicial} subset 
		%$N_{\on{comm}}(G)\subset N(G)$  
{$N(\ZZ,G)\subset NG$}		
		given by
		\[
		N(\ZZ,G)_n=\{(g_1,\ldots,g_n)\mid [g_i,g_j]=1,\;\; \forall 1\leq i,j\leq n\}.
		\]
		It is not hard to see that the simplicial identities respect this commutativity condition, and so this defines a simplicial subset. Note that $N(\ZZ,-)$ is represented by the cosimplicial group $[n]\mapsto \ZZ^{\times n}$. 
		
		A variant of this construction is the \emph{$d$-torsion commutative nerve} for $d\geq 2$. This is the subset $N(\ZZ_{/d},G)\subset N(\ZZ,G)$ given by 
		\[
		N(\ZZ_{/d},G)_n=\{(g_1,\ldots,g_n)\in N(\ZZ,G)_n\mid g_i^d=1,\; \forall 1\leq i\leq n\}.
		\]
		Note that $N(\ZZ_{/d},-)$ is represented by the cosimplicial group $\ZZ_{/d}^{\times n}$. 
	\end{definition}

	\begin{lemma}\label{lem:comm_nerve_2cosk}
		For any group $G$ and any $d\geq 2$, the simplicial sets $N(\ZZ_{/d},G)$ and $N(\ZZ,G)$ are 2-coskeletal. 
	\end{lemma}
	
	\begin{proof}
		We provide the proof for $N(\ZZ,G)$, the proof for $N(\ZZ_{/d},G)$ is identical. 
		
		Suppose given the boundary of a $3$-simplex in $N(\ZZ,G)$, with spine $(g_1,g_2,g_3)$. It will suffice to show that the elements $g_i$ pairwise commute. By hypothesis,  we also have $[g_1g_2,g_3]=1$. Thus, 
		\[
		g_3(g_1g_2)= (g_1g_2)g_3
		\]
		since $g_2$ and $g_3$ also commute by hypothesis, the right-hand side is equal to  
		\[
		g_1g_3g_2.
		\]
		Multiplying on the left by $g_2^{-1}$ shows that $g_3g_1=g_1g_3$, as desired. 
		
		Now suppose given an $n$-simplex boundary in $N(\ZZ_{/d},G)$ for $n>3$, with spine $(g_1,\ldots, g_n)$. Any $g_i$ and $g_j$ are contained in at least one 3-simplex, and thus commute, completing the proof. 
	\end{proof}

{In this work, we are mainly interested in the case where $G$ is the unitary group $U(\hH)$ acting on a finite-dimensional Hilbert space. Then these simplicial sets can be seen as simplicial measurements in the sense of \cite{sdist}.}

\subsection{{Simplicial distributions and measurements}}

{Quantum distributions and more generally non-signaling distributions \cite{barrett2005nonlocal} can be studied using the simplicial framework introduced in \cite{sdist}. The main focus of this investigation is to study Bell's non-locality, and its generalization quantum contextuality. Quantum distributions can be included in this picture as distributions coming from simplicial measurements.}

\begin{definition}[\cite{foulis1994effect}]\label{def:multiplicable effect}
For a partial monoid \( M \), let \( M_n \subset M^{\times n} \) be \emph{the set of \( n \)-multiplicable elements}, {defined recursively as follows:  
\begin{itemize}
\item Every element of $M$ is $1$-multiplicable.

\item An $n$-tuple $(a_1,\cdots,a_n)$ is $n$-multiplicable if $(a_1,\cdots,a_{n-1})$ is $(n-1)$-multiplicable and the pair $(a_1+\cdots+a_{n-1},a_n)$ is $2$-multiplicable.
\end{itemize}} 
\noindent
For an effect algebra $E$, a subset $S\subset E$ with $|S|=n$ will be called \emph{multiplicable} if the \( n \)-tuple of elements obtained by any ordering of the elements of \( S \) belongs to \( E_n \). Note that this property holds for all orderings or none. 
\end{definition}

{A more general version of this definition  will appear in Definition \ref{def:multiplicable effect}.}

%For our final example, fix a finite-dimensional Hilbert space $\scr{H}$, and denote by $\on{Proj}(\scr{H})$ the set of projectors on $\scr{H}$. 

\begin{definition}\label{def:EX}
	For {an effect algebra $E$ and} a set $S$, define 
	\[	
	E(S):= \{ \phi:S \to E\;\mid\; \phi\; \on{fin. supp.},\; \set{\phi(s)}_{s\in S} \text{  multiplicable, } \sum_{s\in S}\phi(s)=1\}.
	\]
	Note that, because of the finite support condition, this yields a functor 
	\[
	\begin{tikzcd}
		E:&[-3em] \Set \arrow[r] & \Set 
	\end{tikzcd}
	\]
	with $E$ acting on maps of sets by taking the sum over fibres. As such, we can compose simplicial sets $X:\Delta^\op \to \Set$ with $E$, yielding a functor which we abusively also denote by 
	\[
	\begin{tikzcd}
		{E}: &[-3em] \Set_\Delta \arrow[r] & \Set_\Delta. 
	\end{tikzcd}
	\] 
Given an effect algebra $E$ and a simplicial set $X$, the simplicial set $E(X)$ is defined to be the composition $E\circ X$.
% denote by $E(X)$ whose $n$-simplices are given by functions $\phi:X_n\to E$ with finite support such that $\set{\phi(x):x\in X_n}$ is multiplicable. The simplicial structure maps are obtained from those of $X$ by the formula: For an ordinal map $\theta:[m]\to [n]$ and $\phi\in E(X)_n$,
%\[
%\theta^*\phi(x) = \sum_{x'\in (\theta^*)^{-1}(x)} \phi(x').
%\]
\end{definition}

%\comm{we also have $E(Y)$ with $E=[0,1]$ denoted by $D(Y)$. Define simplicial distributions, measurements and the simplicial born rule to motivate that $P_\hH(Y)$ behaves like an effect.} 

Main examples come from the theory of simplicial distributions and measurements introduced in \cite{sdist}.
 
\begin{example}\label{ex:projector}
{Recall} the effect algebras $[0,1]$ and $\Proj(\hH)$ introduced in Example \ref{ex:example effect algebras}. 
\begin{enumerate}
\item For $E=[0,1]$, we will write $D$ for the resulting functor. Given a simplicial set \(X\), the \(n\)-simplices of \(D(X)\) are exactly the probability distributions with finite support on \(X_n\). Concretely,
\[
D(X)_n \;:=\; \{\,p : X_n \to [0,1] \,|\,
p \text{ fin. supp., } \sum_{x \in X_n} p(x) \;=\; 1\}.
\]
This simplicial set comes with a canonical map $\delta:X\to D(X)$ that sends a simplex to the delta-distribution peaked at that simplex.

\item When $E=\Proj(\hH)$ we will write $P_\hH$
 for the functor associated with \( E=\Proj(\mathcal{H}) \). 
This time the resulting simplicial set $P_\hH(Y)$ has $n$-simplices given by projective measurements with outcome set $X_n$. 
More explicitly, we have
\[
P_{\scr{H}}(X)_n:=\{\Pi:X_n\to \on{Proj}(\scr{H})\mid P \text{ fin. supp., }\sum_{x\in X_n} \Pi(x)=\mathbb{1}\}.
\]
Similarly, $\delta:X\to P_\hH(X)$ sends a simplex to the projective measurement taking value $\one$ at that simplex.
\end{enumerate}
\end{example}

In the theory of simplicial distributions, measurements and outcomes are represented by simplicial sets.  
\begin{itemize}
\item
A simplicial distribution on the pair consisting of a measurement space $X$ and an outcome space $Y$ is a simplicial map
\[
	\begin{tikzcd}
		X \arrow[r,"p"] & D(Y).
	\end{tikzcd}
\] 

\item 
A simplicial measurement is a simplicial map 
\[
	\begin{tikzcd}
		X \arrow[r,"\Pi"] & P_\hH(Y).
	\end{tikzcd}
\] 
\end{itemize}
Any density operator induces a natural transformation
\[
	\begin{tikzcd}
		\rho_*: &[-3em] P_\hH \arrow[r] & D
	\end{tikzcd}
	\]
defined by a simplicial map $(\rho_*)_Y:P_\hH(Y)\to D(Y)$ that sends a projective measurement $\Pi:Y_n\to \Proj(\hH)$ to the probability distribution $p:Y_n\to [0,1]$ given by the trace formula $p(y)=\Tr(\rho \Pi(y))$. Thus, a simplicial measurement can be turned into a simplicial distribution using the composition:
\[
	\begin{tikzcd}
	p: X \arrow[r,"\Pi"] & P_\hH(Y) \arrow[r,"\rho_*"] &  D(Y).
	\end{tikzcd}
	\]
This way quantum distributions can be studied within the framework of simplicial distributions.	

We can also reformulate Gleason's theorem (Theorem \ref{thm:Gleason}) using this simplicial language.
Let $S^1$ denote the simplicial circle defined to be the quotient $\Delta^1/\partial \Delta^1$. Then for any finite-dimensional Hilbert space of dimension $\geq 3$ the simplicial set maps  $p$ making the diagram commute
\[
\begin{tikzcd}[column sep=huge,row sep=large]
S^1 \arrow[d,"\delta"'] \arrow[rd,"\delta"] & \\
P_\hH(S^1) \arrow[r,"p"] & D(S^1)
\end{tikzcd}
\]
are of the form $p=\rho_*$ for some density operator $\rho$. See \cite[Theorem 7.1]{sdist} for details.

%When applying $P_{\scr{H}}$ to the nerve of $\ZZ_{/d}$, however, this reduces to a familiar example. 

Another important relationship is between commutative nerve spaces and simplicial measurements; \cite[Proposition 6.3]{sdist}.

\begin{proposition}\label{prop:U(H)_P(H)}
	The spectral decomposition induces an isomorphism of simplicial sets 
	\[
	\begin{tikzcd}
		N(\ZZ_{/d},U(\scr{H}))\arrow[r,"\cong"] & P_{\scr{H}}(N(\ZZ_{/d})).
	\end{tikzcd}
	\] 
\end{proposition}

\subsection{{Effect algebroids and cyclic sets}}

An effect algebroid \cite{Roumen} is a straightforward multi-object generalization of an effect algebra.

\begin{definition}\label{def:effect algebroid}
	An \emph{effect algebroid} $E$ consists of 
	\begin{itemize}
		\item A set $E_0$ of points. 
		\item A set $E_1$ of arrows.
		\item Maps $s,t:E_1\to E_0$. For $x,y\in E_0$, we denote by 
		\[
		\Hom(x,y):=\{x\}\times^{s}_{E_0}E_1\times^{t}_{E_0} \{y\}. 
		\] 
		\item A map $u:E_0\to E_1$ which sends $x$ to $0_x$, such that $s\circ u=\on{id}=t\circ u$.
		\item Partial functions 
		\[
		\begin{tikzcd}
			\circ: &[-3em]\Hom(y,z)\times \Hom(x,y)\arrow[r] & \Hom(x,z) 
		\end{tikzcd}
		\]
		\item Involutions 
		\[
		\begin{tikzcd}
			(-)^\perp: &[-3em] \Hom(x,y) \arrow[r] & \Hom(y,x) 
		\end{tikzcd} 
		\]
		We write $1_x:=0_x^\perp$. 
	\end{itemize}
	These data must satisfy:
	\begin{enumerate}
		\item The composition is associative in that $f\circ (g\circ h)$ exists if and only if $(f\circ g)\circ h$ does, and if they exist both are equal. 
		\item The elements $0_x$ are units for composition in that, for $f\in \Hom(x,y)$, $f\circ 0_x$ and $0_y\circ f$ are defined, and both equal $f$. 
		\item (Orthocomplement) The following three conditions on $f\in \Hom(x,y)$ and $g\in \Hom(y,x)$ are equivalent. 
		\begin{itemize}
			\item $f\circ g=1_y$. 
			\item $f=g^{\perp}$.
			\item $g=f^\perp$. 
		\end{itemize}
		\item (Zero-in-one) If $f\circ 1_x$ is defined, $f=0_x$. If $1_y\circ f$ is defined, then $f=0_y$. 
	\end{enumerate}
A \emph{morphism of effect algebroids} $f \colon E \to F$ consists of:
\begin{enumerate}
  \item a function $f \colon E_0 \to F_0$, and
  \item functions $\Hom(x,y) \to \Hom\bigl(f(x), f(y)\bigr)$
\end{enumerate}
such that $f$ preserves $0_x$, $1_x$, and complements, and also satisfies the following condition:
whenever $a \cdot b$ is defined in $E$, then $f(a) \cdot f(b)$ is defined in $F$, and in that case
\[
  f(a \cdot b) \;=\; f(a) \cdot f(b).
\]
We will denote the category of effect algebroids by $\sf{EffAlgd}$.
\end{definition}

{As noted by Roumen, effect algebroids and, in particular, effect algebras, are most naturally studied as cyclic sets, thereby creating a connection to the simplicial framework we are pursuing.}

\begin{definition}
	The \emph{cyclic category} $\Lambda$ has as its objects $\langle n\rangle$ for $n\geq 0$. A morphism $f:\langle n\rangle \to \langle m\rangle$ is an equivalence class of monotone maps of sets
	\[
	\begin{tikzcd}
		f:&[-3em] \ZZ\arrow[r] & \ZZ 
	\end{tikzcd}
	\] 
	such that $f(i+n+1)=f(i)+m+1$ under the equivalence relation that $f\sim g$ when there exists $c\in \ZZ$ such that $f-g=c(m+1)$.  A \emph{cyclic set} is a functor $X: \Lambda^\op \to \Set$, and we denote the category of cyclic sets by $\Set_\Lambda$.  
\end{definition}

{
\begin{example}
We will denote the representable cyclic set $\Lambda(-,\langle n \rangle)$ by $\Lambda_n$.
\end{example}
}

There is a canonical functor 
\[
\begin{tikzcd}
	\Delta\arrow[r] & \Lambda 
\end{tikzcd}
\]
which is bijective on objects and faithful. This functor sends $f:[n]\to [m]$ to (the equivalence class of) 
\[
\tilde{f}(i)= r+ q(m+1) 
\]
where $f(i)=r+q(n+1)$ for $0\leq r\leq n$. Restriction along this functor yields a functor 
\[
\begin{tikzcd}
	\Set_\Lambda \arrow[r] & \Set_\Delta. 
\end{tikzcd}
\]
We will often neglect this functor in notation, and instead speak of the \emph{underlying simplicial set} of a cyclic set $X$, effectively treating $\Delta$ as a subcategory of $\Lambda$.

	The cyclic category may instead by presented by generators and relations. The generators are the face and degeneracy maps of the simplex category,
	\[
	\begin{tikzcd}
		\delta_i: &[-3em] \langle n-1\rangle \arrow[r] & \langle n\rangle 
	\end{tikzcd}
	\]
	and
	\[
	\begin{tikzcd}
		\sigma_i: &[-3em] \langle n+1\rangle \arrow[r] & \langle n\rangle 
	\end{tikzcd}
	\]
	for $0\leq i\leq n$, together with the generating automorphisms
	\[
	\begin{tikzcd}
		\tau_n: &[-3em] \langle n\rangle \arrow[r] & \langle n\rangle 
	\end{tikzcd}
	\] 
	for $n\geq 1$. The relations these must satisfy are the simplicial relations among the $\delta_i$'s and $\sigma_j$'s, together with the relations
	\[
	\begin{aligned}
		\tau_n\circ \delta_i&=\begin{cases}
			\delta_{i-1}\circ \tau_{n-1} & 1\leq i\leq n\\
			\delta_n & i=0
		\end{cases} \\
		\tau_n\circ \sigma_i &= \begin{cases}
			\sigma_{i-1}\circ \tau_{n+1} & 1\leq i \leq n\\
			\sigma_n\tau^2 & i=0 
		\end{cases}\\
		\tau_n^{n+1}&= \on{id}_{\langle n\rangle }.
	\end{aligned}
	\]
	See \cite[Ch. 6]{Loday} for details.

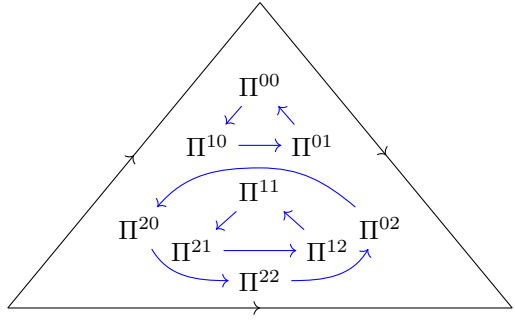
\begin{figure}
	\begin{center}
		\begin{tikzpicture}[decoration={
				markings,
				mark=at position 0.5 with {\arrow{>}}}]
			\path (-25:3.7) node (A2) {}; 
			\path (90:2.5) node (A1) {};
			\path (205:3.7) node (A0) {};
			\draw[postaction={decorate}] (A0.center) to (A1.center);
			\draw[postaction=decorate] (A1.center) to (A2.center);
			\draw[postaction={decorate}] (A0.center) to (A2.center);
			\path (0,1.4) node (B00) {\small{$\Pi^{00}$}};
			\path (0,0) node (B11) {\small{$\Pi^{11}$}};
			\path (0,-1.2) node (B22) {\small{$\Pi^{22}$}};
			\path (-0.7,0.6) node (B10) {\small{$\Pi^{10}$}};
			\path (0.7,0.6) node (B01) {\small{$\Pi^{01}$}};
			\path (-0.9,-0.8) node (B21) {\small{$\Pi^{21}$}};
			\path (0.9,-0.8) node (B12) {\small{$\Pi^{12}$}};
			\path (-1.6,-0.5) node (B20) {\small{$\Pi^{20}$}};
			\path (1.6,-0.5) node (B02) {\small{$\Pi^{02}$}};
			\begin{scope}[blue]
				\draw[->] (B00) to (B10); 
				\draw[->] (B10) to (B01); 
				\draw[->] (B01) to (B00);
				\draw[->] (B11) to (B21); 
				\draw[->] (B21) to (B12); 
				\draw[->] (B12) to (B11);
				\draw[->] (B22) to[out=0,in=-120] (B02); 
				\draw (B02) to[out=140,in=0] (0,0.3);
				\draw[->] (0,0.3) to[out=180,in=50] (B20); 
				\draw[->] (B20) to[out=-60,in=180] (B22);
			\end{scope}
		\end{tikzpicture}
	\end{center}
	\caption{The action of the cyclic automorphism $\tau_2$ determined by the central element $0\in \ZZ_{/3}$ on 2-simplices of $P_{\scr{H}}(N(\ZZ_{/3}))$.} \label{fig:cyclic-structure}
\end{figure}

\begin{example} The main examples of interest for us are the commutative nerves and simplicial measurements with outcome space $N\ZZ_{/d}$.
{It is also worth noting that  on $2$-coskeletal simplicial sets, such as nerves of categories and commutative nerves, a cyclic structure   is uniquely determined by the automorphisms $\tau_1$ and $\tau_2$.}
\begin{enumerate}   
\item 	Of particular interest to us, nerves of groups always carry cyclic structures, in the following manner. Let $G$ be a group, and $z\in Z(G)$ be a central element of $G$. Define maps 
\[
\begin{tikzcd}[row sep=0em]
	\tau_n: &[-3em] N(G)_n \arrow[r] & N(G)_n \\
	& (g_1,\ldots,g_n) \arrow[r,mapsto] & (z(g_1\cdots g_n)^{-1}, g_1,\ldots, g_{n-1}) 
\end{tikzcd}
\]
We can compute 
\[
\tau_n^2(g_1,\ldots,g_n)=(zg_n z^{-1},z(g_1\cdots g_n)^{-1},\ldots,g_{n-1})
\]
and, iterating, 
\[
\tau_n^{n+1}(g_1,\ldots,g_n)=(zg_1z^{-1},\ldots,zg_nz^{-1})=(g_1,\ldots,g_n),
\]
where the final equality follows from the fact that $z$ is central. Indeed, on can verify all of the cyclic identities, and show that this puts a cyclic structure on $N(G)$ (see \cite[\S 7.3.3]{Loday} for further details).

\item Notice that, since $z\in G$ is central, the same construction defines a cyclic structure on $N(\ZZ,G)$. Similarly, if $z$ is a $d$-torsion element, then the construction above defines a cyclic structure on $N(\ZZ_{/d},G)$.
{When $G=U(\hH)$ we endow $N(\ZZ_d,U(\hH))$ with the cyclic structure associated to the central element $z=\omega\one$ where  $\omega=e^{2\pi i/d}$.} 

\item  %\wscomm{I didn't catch this in the previous version, but the change from $\ZZ_{/2}$ to $\ZZ_{/d}$ renders the writing of this example extremely misleading. On $\ZZ_{/d}$ we obtain $d$ distinct cyclic structures, corresponding to the central elements $0,1,\ldots,d-1$. In the special case $d=2$, there are precisely two such structures. I think the example needs to be rewritten in one of the following two  ways:
%\begin{itemize}
%	\item Rewrite this entire example in the special case $\ZZ_{/2}$ (my preferred solution). 
%	\item Write the general formula using an element $z\in \ZZ_{/d}$, never specialize to the case $d=2$, and only give the explicit formulae for $z=1$. This has the benefit of tying directly to the key example (where $d=3$ and $z=1$). 
%\end{itemize}
% }
% \comm{CO: For now I reverted back to $d=2$. In the next iteration we should write it for arbitrary $d$ following the second approach.}
We obtain two distinct cyclic structures on {$N(\ZZ_{/{2}})$} corresponding to the central elements $0$ and $1$. The automorphisms $\tau_1$ and $\tau_2$ associated to these are:
	\begin{itemize}
		\item For the central element $0\in \ZZ_{/2}$, 
		\[
		\begin{aligned}
			\tau_1(k) & = -k\\
			\tau_2(k_1,k_2) & = (-k_2-k_1,k_1).
		\end{aligned}
		\]
		\item For the central element $1\in \ZZ_{/2}$, 
		\[
		\begin{aligned}
			\tau_1(k) & = 1-k\\
			\tau_2(k_1,k_2) & = (1-k_2-k_1,k_1).
		\end{aligned}
		\]
	\end{itemize} 
%\end{example}
  Since $P_{\scr{H}}$ is a functor, composing $P_{\scr{H}}$ with a cyclic set $X:\Lambda^\op\to \Set$ yields a cyclic set. As such, each central element $z\in Z(G)$ determines a cyclic structure on $P_{\scr{H}}(N(G))$. 
%\begin{example}  
	The two cyclic structures on $N(\ZZ_{/2})$ define two cyclic structures on $P_{\scr{H}}(N(\ZZ_{/2}))$. These can be explicitly written as follows, where, for a simplex $\Pi:\ZZ_{/2}^n\to \on{Proj}(\scr{H})$, we write $\Pi^{a_1,\ldots, a_n}$ for the value of $\Pi$ on $(a_1,\ldots,a_n)$. %{Then for $d=2$ we can describe the cyclic structure as follows:}
	\begin{itemize}
		\item For the central element $z=0$, the first two levels of the cyclic structure on $P_{\scr{H}}(N(\ZZ_{/2}))$ are given by 
		\[
		\tau_1(\Pi^0,\Pi^1)=(\Pi^0,\Pi^1)
		\]
		and 
		\[
		\tau_2(\Pi^{00},\Pi^{01},\Pi^{10},\Pi^{11})=(\Pi^{00},\Pi^{11},\Pi^{01},\Pi^{10}). 
		\]
		\item For the central element {$z=1$}, the first two levels of the cyclic structure on $P_{\scr{H}}(N(\ZZ_{/2}))$ are given by 
		\[
		\tau_1(\Pi^0,\Pi^1)=(\Pi^1,\Pi^0)
		\]
		and 
		\[
		\tau_2(\Pi^{00},\Pi^{01},\Pi^{10},\Pi^{11})=(\Pi^{01},\Pi^{10},\Pi^{00},\Pi^{11}). 
		\]
	\end{itemize}
{The latter choice coincides with the cyclic structure on $N(\ZZ_{/2},U(\hH))$ under the isomorphism in Proposition \ref{prop:U(H)_P(H)}.} {Figure \ref{fig:cyclic-structure} represents the cyclic structure of $P_\hH N\ZZ_{/3}$ in degree $2$. We will need it in the construction of our key example; see Construction \ref{const:key example}.}	
\end{enumerate}
\end{example}

{In his thesis \cite{Roumen}, Roumen defines a functor}
\[
\begin{tikzcd}
C: \sf{EffAlgd} \arrow[r] & \sf{Set}_\Lambda
\end{tikzcd}
\]
by sending an effect algebroid $E$ to the cyclic set where $C(E)_n=\sf{EffAlgd}(\Lambda_n,E)$.

\begin{theorem}\label{thm:characterization fo EAd}
The functor $C$ is fully faithful and its essential image consists of those cyclic sets $X$ satisfying the $2$-Segal condition and which sends (U) to a sub-pullback and (Z) to a pullback: 
$$
(U):\begin{tikzcd}[column sep=huge,row sep=large]
{[0]} \arrow[r,"{d^0}"] \arrow[d,"{d^1}"] & {[1]} \arrow[d,"{d^2}"] \\
{[1]} \arrow[r,"{d^0}"] & {[2]}
\end{tikzcd}
\;\;\;\;
\;\;\;\;
\;\;\;\;
(Z):\begin{tikzcd}[column sep=huge,row sep=large]
{[1]} \arrow[r,"{s^0}"] \arrow[d,"{d^1}"] & {[0]} \arrow[d,equal] \\
{[2]} \arrow[r,"{s^0\circ s^0}"] & {[0]}
\end{tikzcd}
$$
\end{theorem}

In the original formulation of \cite{Roumen}, the $2$-Segal condition is implemented by requiring $X$ to send certain kinds of pushouts, referred to as type (A), to pullbacks.

\begin{remark}
	In \cite{SternCY}, it is shown that cyclic 2-Segal sets correspond to Calabi-Yau algebras in spans of sets. Since Calabi-Yau algebras can be loosely thought of as ``coherently associative algebras with 2-sided duals,'' the relation of the orthocomplement to cyclic-ness is unsurprising. We will explore this further in {Section \ref{sec:orthocomplement}.} 
	%later sections. 
\end{remark}
 
% \comm{@Walker: Could there be a way to merge this comment with the remark above:
 
%In Roumen's thesis \cite{Roumen}, effect algebroids are characterized as spiny cyclic sets satisfying the 2-Segal conditions 
%and one additional pullback condition. This is not surprising, given the correspondence between 2-Segal simplicial sets and categories with partially and multiply defined composition laws (so-called $\mu$-categories) proven in \cite[\S 3.3]{DK}. Indeed, restricting the correspondence proven by Dyckerhoff and Kapranov in loc. cit. to the spiny simplicial sets retrieves all parts of the definition of an effect algebroid other than the orthocomplement and the conditions involving it. 

%Since, however, the orthocomplement arises from an additional \emph{cyclic} structure, we begin this section by briefly reviewing Connes' cyclic category $\Lambda$, its relation to the simplex category, and the relation of the 2-Segal conditions to $\Lambda$. 
%}

\section{Three notions of weak associativity}
\label{sec:weak associativity}

{To develop a simplicial theory of effects that encompasses both the natural examples of effect algebras and effect algebroids \cite{foulis1994effect,Roumen} and the constructions appearing in the theory of simplicial distributions and measurements \cite{sdist}, our first step is to formulate a suitable notion of weakened associativity. This weak associativity will be central to our simplicial effect theory and also provides an interesting connection to the partial groups introduced by Chermak \cite{Chermak}.}

There are a number of definitions of partial monoids and partial groups in the literature. Most commonly,  one takes the idea that a partial monoid/group is a set {$M$} with $n$-fold partial multiplication operations defined on subsets {$M_n\subset M^{\times n}$}, of {multiplicable $n$-tuples} {(Definition \ref{def:multiplicable effect})}, subject to additional compatibility conditions. In this schema, one is tempted to take ``associativity'' to mean that ``if $a\cdot b$ is defined and $(a\cdot b)\cdot c$ is defined, then $(b\cdot c)$ is defined, $a\cdot (b\cdot c)$ is defined, the three-fold product $a\cdot b\cdot c$ is defined, and the three are equal.''\footnote{The intuitive definition would also involve an implication starting from $a\cdot (b\cdot c)$ being defined. However we omit this for brevity, since we will take a different definition in the end.} However, in the presence of an invertibility requirement (i.e., when studying \emph{partial groups}), this notion of associativity forces the multiplication to be globally defined, i.e., returns us to the realm of classical group theory. 

One traditional way of resolving this difficulty (e.g., in \cite[\S 2]{Chermak}) is to only require downwards implications for associativity. For instance, if the threefold multiplication $a\cdot b\cdot c$ is defined, then one requires that $(a\cdot b)\cdot c$ and $a\cdot(b\cdot c)$ are defined and equal to the threefold multiplication, but not vice-versa. 

%This is, of course, valid, but 
{We will take a different route}
as we will see our guiding examples satisfy a stronger notion of associativity (we actually term this notion \emph{weak associativity}, because of its relation to the intuitive guess at an associativity condition), which lies between the two definitions described above. Namely, if $(a\cdot b)\cdot c$ and  $a\cdot (b\cdot c)$ are \emph{both} defined then $a\cdot b \cdot c$ is defined and all three are equal.

This gives us our three notions of associativity for partial monoids. In order of decreasing strength: 
\begin{enumerate}
\item  the intuitive definition given in \cite{Segal},

\item the weak associativity described in the previous paragraph, 

\item the associativity imputed to partial groups in \cite{Chermak}. 
\end{enumerate}
We will give formal definitions of each of these in the next section, before exploring their implications further.

Following this exposition, we will explore simplicial sets associated to these three stuctures in the form of nerve operations, and will define conditions which characterize our three forms of associativity in the context of simplicial sets.  

\subsection{Partial monoids, {associativity, and weak associativity}}

In the particular cases we are concerned with, all of the multiplications of a partial monoid are determined by the 2-fold products{, e.g., as in Definition \ref{def:partial monoid}}. We will therefore begin from the assumption that a partial monoid consists of a set, a partial binary multiplication, and conditions thereupon. This differs in form from, e.g., \cite[Definition 2.1]{Chermak}, where multiplications of all arities are part of the initial structure. However, our definition is not as different as it may first appear, as the definition of \cite{Chermak} requires compatibility of the higher arity multiplications with the 2-fold multiplication. Indeed, as we will see, we can build all of definitions up from the {following}, extremely basic structure.

\begin{definition}
	A \emph{partial unital magma} $(M,\cdot,1)$ consists of a set $M$, a partially defined binary operation $\cdot$ on $M$, and an element $1\in M$. We will call a pair $(m,n)\in M\times M$ \emph{multiplicable} if it is in the domain of definition of $\cdot$. These data must satisfy the following \emph{unitality condition}:
	\begin{itemize}
		\item (Unitality) For every $m\in M$, the pairs $(m,1)$ and $(1,m)$ are multiplicable, and $m\cdot 1=m=1\cdot m$. 
	\end{itemize}
	A \emph{morphism of partial unital magmas} is a map of sets which preserves products and units, where preservation of products means that if $(m,n)$ is multiplicable, then $(f(m),f(n))$ is multiplicable, and $f(m\cdot n)=f(m)\cdot f(n)$. We denote the category of partial unital magmas by $\sf{Mag}$.
\end{definition}

We then introduce our first two notions of associativity. The first is that of \cite{Segal} and corresponds to the \emph{2-Segal conditions} of \cite{DK}, but is slightly too strong for our purposes. The second is a strategic weakening of this condition to {accommodate} our key examples. Once this is done, we will discuss the definition given in Chermak, and some key examples.

\begin{definition}
	For $n\geq 2$, a \emph{(binary) bracketing of $n$} is an isomorphism class of planar binary rooted trees with $n$ leaves. Given a bracketing $T$ of $n$ and a partially defined map of sets $\mu:M\times M\to M$, there is a unique corresponding partially defined operation 
	\[
	\begin{tikzcd}
		\mu^T:&[-3em] M^{\times n} \arrow[r] & M 
	\end{tikzcd}
	\] 
	which performs the operation $\mu$ for each vertex of the tree $T$.  
\end{definition}

\begin{definition}\label{def:multiplicable}
	Let $(M,\cdot,1)$ be a partial unital magma. A \emph{bracketed $n$-tuple} in $M$ is an element $\underline{b}\in M^{\times n}$ together with a binary bracketing $T$ of $n$. We say that a bracketed tuple is \emph{multiplicable} in $M$ if $\underline{b}$ is in the domain of definition of the map 
	\[
	\begin{tikzcd}
		\cdot^T: &[-3em] M^{\times n} \arrow[r] & M. 
	\end{tikzcd}
	\]
	We say that an element $\underline{b}\in M^{\times n}$ is \emph{multiplicable} if, for every binary bracketing $T$ of $n$, $(\underline{b},T)$ is multiplicable. 
\end{definition}

\begin{definition}
	Let $(M,\cdot,1)$ be a partial unital magma. We say that $M$ is  
	\begin{itemize}
		\item \emph{weakly associative} if, whenever $\underline{b}\in M^{\times n}$ is multiplicable, then the products associated to any two binary bracketings are equal;
		\item \emph{associative} if it is weakly associative and an $n$-tuple $\underline{b}\in M^{\times n}$ is multiplicable if and only if there exists a binary bracketing $T$ of $n$ such that $(\underline{b},T)$ is multiplicable.
	\end{itemize}
	We will call a weakly associative partial unital magma a \emph{weak partial monoid}, and we will call an associative partial unital magma a \emph{partial monoid}.  
\end{definition}

\begin{remark}
It is worth remarking that our terminology is chosen to accord with Segal's definition of partial monoid in \cite{Segal} {(Definition \ref{def:partial monoid})}, which is helpfully recapitulated in \cite[Example 2.1]{BOORS}. That is, our definition of partial monoid agrees with that of \cite{Segal}, and our definition of a weak partial monoid is a weakening of this definition. This creates a minor terminological issue once we begin discussing invertibility, to wit, in our terminology a partial monoid with all elements invertible is \emph{not} the same thing as a partial group as defined in \cite{Chermak}. 
\end{remark}

  \begin{lemma}
  	Let $(M,\cdot,1)$ be a partial unital magma. Then $M$ is a partial monoid if and only if it satisfies the following condition from \cite{Segal}: 
  	\begin{itemize}
  		\item For $a,b,c\in M$, the multiplication $(a\cdot b)\cdot c$ {is defined} if and only if $a\cdot (b\cdot c)$ is defined, and if both are defined, they are equal. 
  	\end{itemize}
  	Moreover, $M$ is a weak partial monoid if and only if it satisfies the following condition:
  	\begin{itemize}
  		\item For $a,b,c\in M$, if both $(a\cdot b)\cdot c$ and $a\cdot (b\cdot c)$ are defined, they are equal. 
  	\end{itemize}
  \end{lemma}
  
  \begin{proof}
  	Both statements amount to the well-known fact that any two binary parenthesizations can be related by a sequence of moves of the form 
  	\[
  	a(bc)\leftrightsquigarrow (ab)c. \qedhere
  	\]
  \end{proof}
  
  \begin{definition}
  	Given a weak partial monoid $(M,\cdot,1)$ and $k\geq 2$, we denote by $M_k\subset M^{\times k}$ the set of \emph{multiplicable $k$-tuples} in $M$. Note that by weak associativity we obtain a unique map 
  	\[
  	\begin{tikzcd}
  		M_k \arrow[r] & M 
  	\end{tikzcd}
  	\]
  	which is equal to $\cdot^T$ for any binary bracketing $T$. 
  	
  	A morphism $f:(M,\cdot,1)\to (L,\ast,e)$ of weak partial monoids is a map of sets $f:M\to L$ such that $f(1)=e$, and such that $f$ satisfies the following two equivalent conditions. 
  	\begin{enumerate}
  		\item For every $k\geq 2$, $f^{\times k}(M_k)\subset L_k$ and the diagram 
  		\[
  		\begin{tikzcd}[column sep=huge,row sep=large]
  			M_k \arrow[d]\arrow[r,"{f^{\times k}}"] & L_k\arrow[d]\\
  			M \arrow[r,"f"'] & L 
  		\end{tikzcd}
  		\]
  		commutes. 
  		\item If $(m,n)\in M_2$, then $(f(m),f(n))\in L_2$ and $f(m\cdot n)=f(m)\ast f(n)$. 
  	\end{enumerate}
  	We denote by $\sf{WPM}$ and $\sf{PM}$ the categories of weak partial monoids and partial monoids, respectively. Denoting by $\sf{Mon}$ the usual category of monoids, note that we have canonical inclusions 
  	\[
  	\sf{WPM}\supset \sf{PM}\supset \sf{Mon}. 
  	\]
  \end{definition}

\subsection{{Chermak's partial monoids}}	Finally, we reformulate the notion of associativity found in \cite{Chermak}. Note that, rather than being a property, this is an additional \emph{structure} on a partial unital magma. We first state the definition as it appears in \cite[Definition 2.1]{Chermak}, omitting the discussion of invertibility, then provide a reformulation in terms of partial unital magmas. Note that our terminology \emph{does not} follow that of \cite{Chermak}.   
	
	\begin{definition}\label{defn:PAS}
		Let $M$ be a set and $\mathbf{W}(M)$ the free monoid on $M$. A \emph{partial associativity structure (PAS)} on $M$ consists of a subset $\mathbf{D}\subset \mathbf{W}(M)$ and a map $\Pi: \mathbf{D}\to M$ satisfying the following four conditions. 
		\begin{enumerate}
			\item For any $m\in M$, the word $m\in \mathbf{D}$, i.e., $M\subset \mathbf{D}$. 
			\item If $u\circ v\in \mathbf{D}$, then $u\in\mathbf{D}$ and $v\in \mathbf{D}$. 
			\item The map $\Pi$ restricts to the identity map on $M\subset \mathbf{D}$. 
			\item If $u\circ v\circ w\in \mathbf{D}$, then $u\circ \Pi(v)\circ w\in \mathbf{D}$ and 
			\[
			\Pi(u\circ v\circ w)=\Pi(u\circ \Pi(v)\circ w).
			\]
		\end{enumerate} 
	A \emph{morphism of PAS's} is a map of sets $f:M\to N$ such that the induced diagram 
	\[
	\begin{tikzcd}[column sep=huge,row sep=large]
		\mathbf{D}_M\arrow[r,"\mathbf{W}(f)"]\arrow[d,"\Pi_M"'] &\mathbf{D}_N\arrow[d,"\Pi_N"] \\
		M\arrow[r,"f"'] & N 
	\end{tikzcd}
	\]
	commutes. We denote by $\sf{PAS}$ the category of PAS's. 
	\end{definition}

\begin{example}
Canonical examples of PAS's, also {called partial monoids in \cite{Chermak}}, arise from the homotopy theory of classifying spaces of groups and from more generalized structures known as \emph{localities}, 
{a subclass of partial groups with well-behaved $p$-local structure.
{They} were originally motivated by 
the Martino-Priddy conjecture,  
now a celebrated theorem due to Oliver \cite{Oliver1, Oliver2}, {which} states that, given a finite group $G$ and a prime $p$, the homotopy type of {the $p$-completed classifying space} $BG^{\wedge}_p$ 
is uniquely determined by the $p$-local fusion system of $G$.

The ``generalized Martino-Priddy conjecture'' says that there is a unique ``classifying space'' associated to every saturated fusion system $\mathcal{F}$, such a space is denoted by $B\mathcal{F}$. In fact, the conjecture asserts the existence and uniqueness, up to isomorphism, of a certain category $\mathcal{L}$ called the \emph{centric linking system of} $\mathcal{F}$ \cite[Definition 1.7]{BrotoLeviOliver}, and $B\mathcal{F}$ is just the $p$-completed geometric realization of $\mathcal{L}$ (i.e. $B\mathcal{F}=\lvert\mathcal{L}\rvert^{\wedge}_p$).

Chermak solved this conjecture \cite[Section 7]{Chermak} by showing that there is a bijective correspondence, up to isomorphism, between the class of centric linking systems and a certain subclass of localities that he also called centric linking systems, and then the existence and uniqueness, up to isomorphism, of a Chermak's centric linking system associated to $\mathcal{F}$.
}
\end{example}

Our reformulation is as follows. 
	
	\begin{definition}
		Let $(M,\cdot, 1)$ be a partial unital magma. We say that a multiplicable $n$-tuple $\underline{b}$ is \emph{associable} if, for any two  bracketings $T,S$ of $n$, $\cdot^T (\underline{b})=\cdot^S (\underline{b})$. We further call $\underline{b}$ \emph{fully associable} if, for any $0\leq i<j\leq n$, the tuple $(b_i,b_{i+1},\ldots,b_j)$ is associable. We denote by $F_n(M)$ the set of fully associable $n$-tuples in $M$, and write 
		\[
		F(M):=\coprod_{n\geq 0} F_n(M).
		\]  
	\end{definition}

	\begin{definition}\label{defn:assoc_datum}
		Let $(M,\cdot, 1)$ be a partial unital magma. An \emph{associativity datum} on $M$ consists of a subset $A_n\subset F_n(M)$ for each $n\geq2$ such that 
		\begin{enumerate}
			\item The set $A_2$ is the domain of the partial multiplication $\cdot$. 
			\item If $(b_1,\ldots,b_n)\in A_n$, then for any $1\leq i<n$, $(b_1,\ldots, b_i)\in A_i$ and $(b_{i+1},\ldots,b_n)\in A_{n-i}$. 
			\item If $(b_1,\ldots,b_n)\in A_n$ , then $(b_1,\ldots,b_{i-1},1,b_i,\ldots,b_n)\in A_{n+1}$ for any $1\leq i\leq n+1$. 
		\end{enumerate} 
		A morphism $f:M\to N$ of partial unital magmas is said to \emph{commute with associativity data} if $f^{\times n}(A_n^M)\subset f^{\times n}(A_n^N)$. We denote by $\sf{Mag}^{\on{ad}}$ the category of partial unital magmas equipped with associativity data.
	\end{definition}

	\begin{example}
		Note that every partial unital magma $(M,\cdot,1)$ can be equipped with an associativity datum by setting $A_n=F_n(M)$. Note that this associativity datum is the maximal associativity datum possible. This defines a functor $\sf{Mag}\to \sf{Mag}^{\on{ad}}$.
	\end{example}

	\begin{proposition}\label{prop:Magad equivalent to PAS}
		There is an equivalence of categories 
		\[
		\sf{Mag}^{\on{ad}}\simeq \sf{PAS}.
		\]
	\end{proposition}

	\begin{proof}
		Given a partial unital magma $M$, define a map 
		\[
		\begin{tikzcd}
			\psi_M: &[-3em] F(M)\arrow[r] & \mathbf{W}(M)
		\end{tikzcd}
		\]
		which is simply the inclusion. Given an associativity datum $A_M$ on $M$, define a PAS on $M$ by setting $\mathbf{D}_{M,A}=\psi_M(A)\cup M\cup \{\varnothing\}$. We define $\Pi_{M,A}$ to be the map which acts on length-$n$ words by $\cdot^T$ for some bracketing $T$ of $n$, acts on $M$ as the identity, and sends $\varnothing$ to $1$.  
		
		We then check conditions (1)-(4) in {Definition} \ref{defn:PAS}. Condition (1) holds by construction. Condition (2) corresponds exactly to condition (2) in {Definition} \ref{defn:assoc_datum}. Condition (3) again follows by construction. Condition (4) is equivalent to the fact that the words in $A$ are fully associable. Note that these implications are reversible, showing immediately that this functor is essentially surjective. 
		
		To see that this construction is functorial, let $f:(M,A_M)\to (N,A_N)$ be a morphism in $\sf{Mag}^{\on{ad}}$. It is easy to check that the same underlying map of sets defines a morphism from $(M,\sf{D}_{M,A_M},\Pi_{M,A_M})$ to $(N,\sf{D}_{N,A_N},\Pi_{N,A_N})$. Since morphisms in both categories are determined by their underlying map of sets, it is immediate that this functor is faithful. Since, modulo  $M$ and $\varnothing$, $\mathbf{D}_{M,A_M}$ and $A_M$  are the same set, and commuting with $\Pi_{M,A_M}$ is implied by commuting with $\cdot$ by fully associability, it is similarly easy to se that this functor is full. 
	\end{proof}

	\begin{remark}\label{rmk:Chermak_non-equivalent_same}
		This proposition and the preceding example go a long way towards explaining why we need a definition which is intermediate between Segal's and Chermak's to define a theory of partial monoids which includes interesting examples of both partial groups and effect algebras. On the one hand, Segal's associativity condition precludes the possibility of interesting partial groups, since adding invertibility forces the multiplication to be globally defined. On the other, as the above shows, Chermak's definition need not require \emph{any} associativity for the underlying binary operation, and can encode the same underlying algebraic structure (partial unital magma) in multiple non-isomorphic ways. 
	\end{remark}

	We conclude this section with a discussion of the relations between the categories of algebraic structures defined in this section. We have obvious forgetful functors which we denote as follows 
	\[
	\begin{tikzcd}
		\sf{PM}\arrow[r,"\Psi"] & \sf{WPM}\arrow[r,"\Phi"] & \sf{Mag} 
	\end{tikzcd}
	\]	
	and a functor 
	\[
	\begin{tikzcd}
		(-)^\sharp:&[-3em] \sf{Mag} \arrow[r] & \sf{Mag}^{\on{ad}}
	\end{tikzcd}
	\]
	which equips a partial unital magma with the maximal associativity datum. The functor $(-)^\sharp$ is right adjoint to the forgetful functor. 
	
	\begin{proposition}\label{prop:max_ass_data_FF}
		The functors $\Psi$, $\Phi$, and $(-)^\sharp$ are fully faithful. 
	\end{proposition}

	\begin{proof}
		For $\Phi$ and $\Psi$, this is immediate, since the definition of morphism in no way depends on the associativity condition. Similarly, every morphism $f:(M,\cdot,1_M)\to (N,\ast,1_N)$ for partial unital magmas commutes with the maximal associativity datum, and so $(-)^\sharp$ is fully faithful. 
	\end{proof}

	\section{Nerves and associativity conditions}
	\label{sec:Nerves and associativity conditions}
	
	We now turn to considering the algebraic structures introduced above as simplicial sets. To this end, we will define a nerve operation from $\sf{Mag}^{\on{ad}}$ which yields nerves for all four of the categories described above. However, since the functors $\Psi$, $\Phi$, and $(-)^\sharp$ are fully faithful, all of these nerves can, in fact, be induced by a cosimplicial object in $\sf{PM}$. 
	
	\begin{definition}
		We define a cosimplicial object in $\sf{PM}$ 
		\[
		\begin{tikzcd}
			\mathfrak{D}^\bullet: &[-3em] \Delta \arrow[r] & \sf{PM} 
		\end{tikzcd}
		\]
		where $\mathfrak{D}^n$ is the partial monoid with non-identity elements $m_{i,j}$ for $0\leq i<j\leq n$. The multiplicable pairs are $(m_{i,j},m_{j,k})$ for $0\leq i<j<k\leq n$, and the multiplication for non-identity elements is 
		\[
		m_{i,j}\cdot m_{j,k}=m_{i,k}. 
		\] 
		Given a map $\phi:[n]\to [k]$ in $\Delta$, the corresponding map $\phi_\ast:\mathfrak{D}^n\to \mathfrak{D}^k$ is given by 
		\[
		\phi_\ast(m_{i,j})=m_{\phi(i),\phi(j)}
		\]
		where we take the convention that $m_{i,i}=1$ is the identity element. It is easily checked that this is a map of weak partial monoids.
		
		Composing with the functors 
		\[
		\begin{tikzcd}
			\sf{PM}\arrow[r,"\Psi"] & \sf{WPM}\arrow[r,"\Phi"]& \sf{Mag}\arrow[r,"(-)^\sharp"] & \sf{Mag}^{\on{ad}} 
		\end{tikzcd}
		\]
		this yields a cosimplicial object in all four categories, which we will abusively also denote by $\mathfrak{D}^\bullet$.
	\end{definition} 

	\begin{definition}
		We denote by 
		\[
		\begin{tikzcd}
			N:&[-3em] \sf{Mag}^{\on{ad}} \arrow[r] & \Set_\Delta 
		\end{tikzcd}
		\]
		the nerve operation corresponding to the cosimplicial object $\mathfrak{D}^\bullet$. Explicitly, 
		\[
		N(M,A)_n:=\sf{Mag}^{\on{ad}}((\mathfrak{D}^n)^\sharp,(M,A)).
		\]
		Abusively, we also use $N$ to denote the functors from $\sf{PM}$, $\sf{WPM}$, and $\sf{Mag}$ to simplicial sets. Note that if $\sf{C}$ is any one of these categories, the fully-faithfulness of the functors relating them implies that 
		\[
		N(M)_n=\sf{C}(\mathfrak{D}^n,M).
		\]
	\end{definition}

	\subsection{Spinyness, reducedness, and {partial unital magmas} with associativity data}

	We begin our discussion by showing that these nerves are fully faithful, and by characterizing the essential image of $\sf{Mag}^{\on{ad}}$ under $N$. For the latter, we will need the following definitions.

		Given a simplicial subset $S\subset \Delta^n$ and an arbitrary simplicial set $X$, denote the corresponding \emph{membrane set} (of \cite[\S 2.2]{DK})
		by 
		\[
		\on{MS}(S,X):= \Set_\Delta(S,X)
		\]
		and note that this provides a functor $\on{MS}(-,X)$ from the opposite of the poset of simplicial subsets of $\Delta^n$ to $\Set$. Two key examples of simplicial subsets of $\Delta^n$ will define the 1-Segal and 2-Segal conditions.   
		
		For $n\geq 2$, denote by $\on{Sp}^n\subset \Delta^n$ the spine of $\Delta^n$. Then the corresponding membrane sets are 
		\[
		\on{MS}(\on{Sp}^n,X)\cong X_1\times_{X_0}X_1\times_{X_0}\cdots \times_{X_0}X_1. 
		\]
		The \emph{1-Segal maps} are the maps 
		\[
		\begin{tikzcd}
			X_n \arrow[r] & \on{MS}(\on{Sp}^n,X). 
		\end{tikzcd}
		\]
		We call a simplicial set $X$ \emph{spiny} if the 1-Segal maps are injective.

		\begin{figure}[htb] 
		  \begin{center}
		  	\begin{tikzpicture}[decoration={
		  			markings,
		  			mark=at position 0.5 with {\arrow{>}}}]
		  		\path (0,0) node (A0) {0};
		  		\path (2,0) node (A1) {1};
		  		\path (2,2) node (A2) {2};
		  		\path (0,2) node (A3) {3};
		  		 
		  		 \draw[postaction={decorate}] (A0) to (A1); 
		  		 \draw[postaction={decorate}] (A1) to (A2);
		  		 \draw[postaction={decorate}] (A2) to (A3);
		  		 \draw[postaction={decorate}] (A0) to (A3);  
		  		 \draw[postaction={decorate}] (A1) to (A3);  
		  		 \begin{scope}[xshift=5cm]
		  		 	\path (0,0) node (B0) {0};
		  		 	\path (2,0) node (B1) {1};
		  		 	\path (2,2) node (B2) {2};
		  		 	\path (0,2) node (B3) {3};
		  		 	
		  		 	\draw[postaction={decorate}] (B0) to (B1); 
		  		 	\draw[postaction={decorate}] (B1) to (B2);
		  		 	\draw[postaction={decorate}] (B2) to (B3);
		  		 	\draw[postaction={decorate}] (B0) to (B3);  
		  		 	\draw[postaction={decorate}] (B0) to (B2);  
		  		 \end{scope}
		  	\end{tikzpicture}
		  \end{center}
		\caption{{Triangulations of the planar $4$-gon with ordered vertices $P_4$. We sometimes denote the former triangulation by $\squaredivright$ and the latter by $\squarediv$.}}
		\label{fig:triangulations}
		\end{figure}
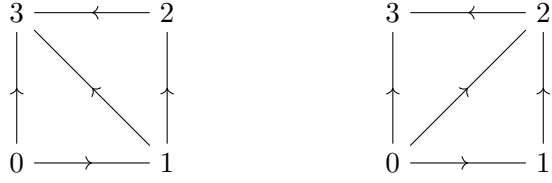

		For $n\geq 3$, let $P_{n+1}$ denote a planar $(n+1)$-gon with vertices labeled $0,1,\ldots,n$ counterclockwise. A triangulation $\scr{T}$ of $P_{n+1}$ determines a $2$-dimensional simplicial subset $\Delta^{\scr{T}}\subset \Delta^n$ consisting of precisely those $2$-simplices $\Delta^{\{i,j,k\}}$ such that $\{i,j,k\}$ is a triangle of $\scr{T}$. {See Figure \ref{fig:triangulations} for $n=3$.} For $X\in \Set_\Delta$, there is a natural identification of 
		\[
		\Set_{\Delta}(\Delta^\scr{T},X)
		\]
		with an iterated pullback of copies of $X_2$ over copies of $X_1$. 	
		
		%The \emph{2-Segal maps} are the natural maps 
		%\[
		%\begin{tikzcd}
	%		X_n \arrow[r] & \on{MS}(\Delta^{\scr{T}},X)
	%	\end{tikzcd}
%		\]
%		induced by triangulations of $P_{n+1}$. 

\begin{definition}	\label{def:2 segal}
We call a simplicial set \emph{2-Segal} if for every triangulation of $P_{n+1}$ the induced natural maps
\[
		\begin{tikzcd}
			X_n \arrow[r] & \on{MS}(\Delta^{\scr{T}},X),
		\end{tikzcd}
		\]
called the 2-Segal maps, are isomorphisms. 
\end{definition} 		
		
		Finally, we call a simplicial set $X$ \emph{reduced} if $X_0\cong \ast$.

	\begin{lemma}\label{lem:Nerve_n_composable}
		A morphism $f:(\mathfrak{D}^n)^\sharp\to (M,A_M)$ of partial unital magmas with associativity data is uniquely determined by its values on the elements $m_{i,i+1}$ for $0\leq i<n$. An assignment sending $m_{i,i+1}\in \mathfrak{D}^n$ to $g(m_{i,i+1})\in M$ extends to a morphism in $\sf{Mag}^{\on{ad}}$ if and only if the sequence 
		\[
		(g(m_{0,1}),g(m_{1,2}),\ldots,g(m_{n-1,n}))
		\]
		lies in $A_M$. 
	\end{lemma}

	\begin{proof}
		Firstly, if $f$ is a morphism in $\sf{Mag}^{\on{ad}}$, then the sequence $(f(m_{0,1}),f(m_{1,2}),\ldots,f(m_{n-1,n}))$ necessarily lies in $A_M$, since $(m_{0,1},\ldots,m_{n-1,n})$ lies in $F_n(M)$. Further, it is immediate that $f$ is determined by its values $f(m_{i,i+1})$, since the fact that $f$ is a partial magma morphism and $(m_{0,1},\ldots,m_{n-1,n})$ is fully associable implies that 
		\[
		f(m_{i,j})=f(m_{i,i+1})\cdot f(m_{i+1,i+2})\cdots f(m_{j-1,j}). 
		\]
		
		On the other hand, suppose we are given an sequence
		\[
		(g(m_{0,1}),g(m_{1,2}),\ldots,g(m_{n-1,n}))\in A_M.
		\]
		 Then define 
		\[
		g(m_{i,j}):=g(m_{i,i+1})\cdot g(m_{i+1,i+2})\cdots g(m_{j-1,j})
		\]
		for $1\leq i+1<j\leq n$ (this is well-defined, as the chosen tuple is fully associable). Then 
		\[
		\begin{aligned}
			g(m_{i,j}\cdot m_{j,k})&= (g(m_{i,i+1})\cdot g(m_{i+1,i+2})\cdots g(m_{j-1,j}))\cdot (g(m_{j,j+1})\cdot g(m_{j+1,j+2})\cdots g(m_{k-1,k}))\\
			& =g(m_{i,i+1})\cdot g(m_{i+1,i+2})\cdots g(m_{j-1,j})\cdot g(m_{j,j+1})\cdot g(m_{j+1,j+2})\cdots g(m_{k-1,k})\\
			&= g(m_{i,k})
		\end{aligned}
		\]
		by full associability. Thus, $g$ is a morphism of partial unital magmas. The fact that it preserves associativity data is immediate from the assumption that $(g(m_{0,1}),\ldots,g(m_{n-1,n}))\in A_M$. 
	\end{proof}

	\begin{corollary}\label{cor:Segal_nerve}
		For a partial unital magma with a $(M,\cdot,1)$, we have $N(M)_k\cong F_k(M)\subset M^{\times k}$. Under this identification, the face and degeneracy maps are given by
		\[
		d_i(m_1,\ldots, m_k)=\begin{cases}
			(m_2,\ldots, m_k) & i=0 \\
			(m_1,\ldots,m_{i}\cdot m_{i+1},\ldots,m_k) & 0<i<k\\
			(m_1,\ldots,m_{k-1}) & i=k
		\end{cases}
		\]
		and 
		\[
		s_i(m_1,\ldots,m_k)= (m_1,\ldots,m_i,1,m_{i+1},\ldots,m_k). 
		\]
	\end{corollary}

	\begin{remark}
		The characterization of the nerve given in {Corollary} \ref{cor:Segal_nerve} is precisely the nerve for partial monoids given in \cite{Segal}, and generalizes with the nerve for partial groups given in the proof of \cite[Theorem 4.8]{Gonzales}.
	\end{remark}
	
	\begin{proposition}\label{prop:ess_im_magad}
		The nerve 
		\[
		\begin{tikzcd}
			N:&[-3em] \sf{Mag}^{\on{ad}} \arrow[r] & \Set_\Delta
		\end{tikzcd}
		\]
		is fully faithful, and its essential image consists of those simplicial sets which are spiny and reduced. 
	\end{proposition}

	\begin{proof}
		The content of this statement is already implicit in the proof of \cite[Theorem 4.8]{Gonzales}, but we spell out the details for completeness. 
		
		Since a morphism $f:(M,A_M)\to (L,A_L)$ in $\sf{Mag}^{\on{ad}}$ is uniquely determined by the map of sets $f:M\to L$, and this map is identified with $N(f)_1:N(M)_1\to N(L)_1$, it is immediate that this functor is faithful. 
		
		Now suppose that $f:N(M)\to N(L)$ is a map between nerves of elements of $\sf{Mag}^{\on{ad}}$. Throughout, we implicitly use the description of the nerve from {Corollary} \ref{cor:Segal_nerve}. Note that, since $f$ commutes with the face maps $d_0$ and $d_n$, applying $f_k$ to a tuple $(m_1,\ldots,m_k)\in A_M$ must yield the tuple $(f_1(m_1),\ldots, f_1(m_k)\in A_L$. Similarly, since $f$ commutes with the face maps $d_i$ for $0<i<k$, we see that $f_k(m_1\cdot m_2\cdots m_k)$ must be equal to $f_1(m_1)\cdot f_1(m_2)\cdots f_k(m_k)$. Since $f$ commutes with the degeneracy map $s_0$,  $f_1(1_M)=1_L$, and thus $f_1$ is a morphism of partial unital magmas, and $f$ is its nerve. It is immediate that $f$ preserves associability data, completing the proof of faithfulness. 
		
		To characterize the essential image, note that, by construction, the nerve of a partial unital magma is spiny and reduced. 
		
		On the other hand, suppose that $X$ is a spiny, reduced simplicial set. Define a partial unital magma with underlying set $X_1$, multiplication 
		\[
		\begin{tikzcd}
			X_1\times X_1 & X_2 \arrow[l,hookrightarrow,"{(d_2,d_0)}"'] \arrow[r,"d_1"] & X_1 
		\end{tikzcd}
		\]
		and unit 
		\[
		\begin{tikzcd}
			\ast \cong X_0 \arrow[r,"s_0"] & X_1.
		\end{tikzcd}
		\]
		Note that defining the partial multiplication makes use of the fact that $X$ is spiny and reduced. We then note that  
		  
		Again making use of spiny-ness and reduced-ness, we identify $X_n$ with a subset of $X_1^{\times n}$. The 1-Segal map $X_n\to X_1\times \cdots\times X_1$ identifies $X_n$ with a set of $n$-tuples in $X_n$. For $0\leq i\leq n$, the commutativity of the diagrams 
		{\footnotesize
			\[
			\begin{tikzcd}
				X_n\arrow[r,equal]\arrow[d]  &[3em] X_n\arrow[r,"d_i"]\arrow[d] &[3em] X_{n-1}\arrow[d]\\
				X_1\times\cdots\times\underbrace{X_1}_{\{i-1,i\}}\times \underbrace{X_1}_{\{i,i+1\}}\times \cdots X_1 & X_1\times\cdots\times \underbrace{X_2}_{\{i-1,i,i\}}\times \cdots X_1\arrow[l,"{(\id,\ldots,(d_2,d_0),\ldots,\id)}"]\arrow[r,"{(\id,\ldots,d_1,\ldots,\id)}"'] & X_1\times\cdots\times\underbrace{X_1}_{\{i,i+2\}}\times \cdots X_1
			\end{tikzcd}
			\]
		}
		and 
		\[
		\begin{tikzcd}
			X_n \arrow[r,"s_i"]\arrow[d] &[3em] X_{n+1}\arrow[d] \\
			X_1\times \cdots \times \underbrace{\ast}_{i}\times \cdots \times X_1\arrow[r,"{(\id,\ldots,s_0,\ldots,\id)}"'] & X_1 \times\cdots \times X_1   
		\end{tikzcd}
		\]
		shows that, under this identification 
		\[
		d_i(x_1,\ldots, x_n) =\begin{cases}
			(x_2,\ldots, x_n) & i=0 \\
			(x_1,\ldots, x_i\cdot x_{i+1},\ldots, x_n) & 0<i<n\\
			(x_1,\ldots,x_{n-1}) & i=n 
		\end{cases}
		\]
		and 
		\[
		s_i(x_1,\ldots,x_n)=(x_1,\ldots, x_{i},1,x_{i+1},\ldots, x_n). 
		\]  
		This characterization of the face maps immediately implies that setting $A_n=X_n$ for $n\geq 2$ means that $A_n\subset F_n(X_1)$, and that condition (2) in {Definition} \ref{defn:assoc_datum} holds. The characterization of the degeneracy maps implies that condition (3) holds. Thus, we have constructed a weak partial monoid $(X_1,\cdot,1)$ with an associativity datum $\coprod_{n\geq 2} X_n$ whose nerve is isomorphic to $X$, completing the proof. 
	\end{proof}
	
\begin{remark}
{
Broto and Gonzales in \cite{broto2021extension} define partial monoids as spiny reduced simplicial sets and assert that this definition is equivalent to Chermak's via the nerve functor. Propositions \ref{prop:Magad equivalent to PAS} and \ref{prop:ess_im_magad} offer an alternative proof for this equivalence.
}
\end{remark}

	\subsection{Associativity and the 2-Segal condition}
	
	We now briefly provide a characterization of the essential image of $\sf{PM}$ under the nerve {functor}.
	% in terms of the \emph{2-Segal conditions} of \cite{DK}. 
	
	\begin{proposition}
			The nerve 
			\[
			\begin{tikzcd}
				N:&[-3em] \sf{PM} \arrow[r] & \Set_\Delta
			\end{tikzcd}
			\]
			is fully faithful, and its essential image consists of those simplicial sets which are spiny, reduced, and 2-Segal. 
	\end{proposition}	
	
	\begin{proof}
		The fact that the nerve is fully faithful follows from the fact that the functor $\sf{PM}\to \sf{Mag}^{\on{ad}}$ is fully faithful together with {Proposition} \ref{prop:ess_im_magad}. For the fact that the nerve of a partial monoid is 2-Segal, see, e.g., \cite[Example 2.1]{BOORS}. 
		
		It thus only remains for us to show that, when $X$ is spiny, reduced and 2-Segal, the construction of the partial magma $(X_1,\cdot,1)$ in the proof of  {Proposition} \ref{prop:ess_im_magad} yields an \emph{associative} partial magma with maximal associativity data $A_n=F_n(X)$. However, in that construction, the lowest-dimensional $2$-Segal conditions require that the maps 
		\[
		\begin{tikzcd}
			X_2\times^{\squareslashright} X_2 & X_3 \arrow[l] \arrow[r] & X_2\times^{\squareslashleft} X_2
		\end{tikzcd}
		\]
		are isomorphisms, where the $\squareslashleft$ superscripts indicate which triangulation the $2$-Segal pullback corresponds to. Using the characterization of the face and degeneracy maps, this immediately implies that $a\cdot (b\cdot c)$ is multiplicable if and only if $(a\cdot b)\cdot c)$ is, and then both are equal. Thus, the partial unital magma defined above is a partial monoid. 
		
		{Let} $\Delta_{\on{inj}}\subset \Delta$ denote the wide subcategory which includes precisely the injective maps.
		Finally, the 2-Segal isomorphisms
	%	\[
%		X_n \to \lim_{\Delta^{\on{inj}}/\Delta^{\scr{T}}} X_i 
%		\]
\[
\begin{tikzcd}
X_n \arrow[r] & \lim_{\Delta^{\on{inj}}/\Delta^{\scr{T}}} X_i 
\end{tikzcd}
\]		
		
		%\comm{TBA: definition of $\Delta^{\on{inj}}$}\wscomm{This should be in the, as-yet-unwritten, simplicial preliminaries section.}
		for \emph{every} triangulation $\scr{T}$ of $P_{n+1}$ shows that the tuples representing elements of $X_n$ are precisely those which are fully associable, proving the desired maximality. 
	\end{proof}

%	\begin{remark}
%		Note that every effect algebra is a partial monoid, hence a reduced, spiny, 2-Segal simplicial set. This is also proven in \cite{Roumen}.  
%	\end{remark}	

	\subsection{Weak associativity and weak 2-Segal conditions}
	
	We now define conditions on a simplicial set corresponding to \emph{weak} associativity. These make use of the same $2$-Segal maps --- effectively because of the correspondence between triangulations of $P_{n+1}$ and binary planar rooted trees (parenthesizations) --- but are strictly weaker, as we will see. 
	
	\begin{figure}[htb] 
	  \begin{center}
	   \includegraphics[width=0.9\textwidth]{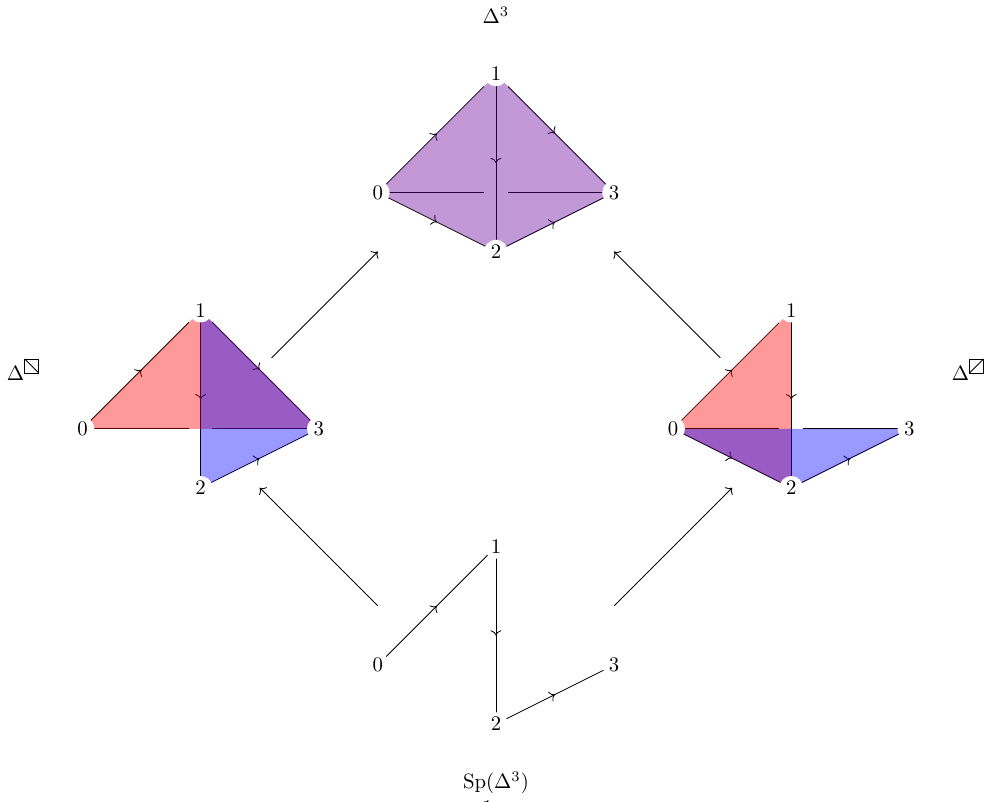}
	  \end{center}
	\caption{{The poset $\scr{I}_3$.}}
	\label{fig:diagram}
	\end{figure}

	\begin{definition}
		Let $n\geq 3$. We denote by $\scr{I}_n$ the full subposet of the poset of simplicial subsets of $\Delta^n$ on the objects $\on{Sp}^n$, $\Delta^n$ itself, and $\Delta^{\scr{T}}$ for any triangulation of {$P_{n+1}$}. For $n\geq 3$ and $X\in \Set_\Delta$, let  
		\[
		\begin{tikzcd}
			\on{MS}(X,n): &[-3em] \scr{I}_n^\op \arrow[r] & \Set 
		\end{tikzcd}
		\] 
		be the functor which takes a simplicial subset of $\Delta^n$ to the corresponding membrane set of \cite[\S 2.2]{DK}.
		% \comm{ref missing}. 
		We call $X\in \Set_\Delta$ \emph{weakly 2-Segal} if, for every $n\geq 3$, the diagram $M(X,n)$ is a limit diagram.
		
		We denote by $\sf{W}2\sf{S}_\Delta\subset \Set_\Delta$ the full subcategory of weakly 2-Segal sets, and by $\sf{W}2\sf{S}^{\on{sp}}_\Delta$ the full subcategory on the spiny weakly 2-Segal sets. Finally, we denote by $\sf{W}2\sf{S}^{\on{sp},\ast}_\Delta$ the full subcategory on the reduced spiny weakly 2-Segal sets. 
	\end{definition}

	\begin{theorem}\label{thm:nerve_characterization}
		The nerve is an equivalence of categories 
		\[
		\begin{tikzcd}
			N :&[-3em] \sf{WPM} \arrow[r,"\simeq"] & \sf{W2S}^{\on{sp},\ast}_\Delta.
		\end{tikzcd}
		\]
	\end{theorem} 
	
	We first show that the nerve factors through $\sf{W2S}^{\on{sp},\ast}_\Delta$. To do this, we denote by 
	\[
	\begin{tikzcd}
		C_{\mathfrak{D}}^n:&[-3em]  \scr{I}_n \arrow[r] & \sf{WPM}  
	\end{tikzcd}
	\] 
	the functor which sends each subset $S\subset \Delta^n$ to the colimit 
	\[
	\colim_{\Delta^{\on{inj}}/S} \mathfrak{D}^\bullet
	\]
	in $\sf{WPM}$. 
	
	\begin{proposition}\label{prop:nerve_W_2Seg}
		For every $n\geq 3$, the functor $C_{\mathfrak{D}}^n$ is a colimit cocone. 
	\end{proposition}
	
	\begin{proof}
		Let $(M,\cdot,1)$ be a partial monoid, and suppose we are given morphisms of weak partial monoids 
		\[
		\begin{tikzcd}
			\mu_S:&[-3em] \colim_{\Delta^{\on{inj}}/S} \mathfrak{D}^\bullet \arrow[r] & M 
		\end{tikzcd}
		\]
		for $S\in \scr{I}_n$ other than $\Delta^n$ itself. We will construct a morphism of weak partial monoids $\mu:\mathfrak{D}^n\to M$ defining a morphism of cones, and show that it is unique. 
		
		Since the diagram 
		\[
		\begin{tikzcd}[column sep=huge,row sep=large]
			\mathfrak{D}^n \arrow[r,"\mu"] & M \\
			\mathfrak{D}^1\amalg\cdots \amalg \mathfrak{D}^1\arrow[ur,"\mu_{\on{Sp}(n)}"'] \arrow[u,"\on{spine}"]
		\end{tikzcd}
		\]
		must commute, we see that $\mu(m_{i,i+1})=\mu_{\on{Sp}(n)}|_{\mathfrak{D}^{i,i+1}}(m_{0,1})$, and thus, by {Lemma}\ref{lem:Nerve_n_composable}, uniqueness will follow by showing that the sequence 
		\[
		\left(\mu_{\on{Sp}(n)}|_{\mathfrak{D}^{0,1}}(m_{0,1}),\ldots, \mu_{\on{Sp}(n)}|_{\mathfrak{D}^{n-1,n}}(m_{0,1})\right)
		\] 
		is $n$-composable. To simplify notation, for $S\subset \Delta^n$ in $\scr{I}_n$, we will write $\mu_S(m_{i,j})$ for the unique element of $C_{\mathfrak{D}}^n(S)$ which maps to $m_{i,j}$ in $\mathfrak{D}^n$. 
		
		Firstly, we note that the $\mu_{\on{Sp}^n}(m_{i,i+1})$ are pairwise composable in the appropriate sequence. This follows since, for any $0\leq i$ such that $i+2\leq n$, the triangle $\Delta^{i,i+1,i+2}$ lies in some triangulation $\scr{T}$ of $P_n$, and so 
		\[
		\mu_{\Delta^{\scr{T}}}(m_{i,i+1})\cdot \mu_{\Delta^{\scr{T}}}(m_{i+1,i+2})=\mu_{\Delta^{\scr{T}}}(m_{i,i+2}). 
		\]
		Since $\mu_{\Delta^{\scr{T}}}(m_{i,i+1})=\mu_{\on{Sp}^n}(m_{i,i+1})$ and $\mu_{\Delta^{\scr{T}}}(m_{i+1,i+2})=\mu_{\on{Sp}^n}(m_{i+1,i+2})$, this shows the desired pairwise composability. 
		
		Now suppose given a binary bracketing of 
		\[
		\left(\mu_{\on{Sp}^n}(m_{0,1}),\ldots, \mu_{\on{Sp}^n}(m_{n-1,n})\right). 
		\]
		There is a corresponding triangulation $\scr{T}$ of $P_n$ whose triangles are precisely the $\Delta^{i,j,k}$ where 
		\[
		\cdots((a)\cdot(b))\cdots 
		\] 
		is a part of the bracketing, $(a)$ is a bracketing of 
		\[
		\left(\mu_{\on{Sp}^n}(m_{i,i+1}),\ldots, \mu_{\on{Sp}^n}(m_{j-1,j})\right)
		\]
		and $(b)$ is a bracketing of 
		\[
		\left(\mu_{\on{Sp}^n}(m_{j,j+1}),\ldots, \mu_{\on{Sp}^n}(m_{k-1,k})\right).
		\]
		Since $\mu_{\Delta^{\scr{T}}}$ is a map of partial monoids, this shows that the original bracketing of 
		\[
		\left(\mu_{\on{Sp}^n}(m_{0,1}),\ldots, \mu_{\on{Sp}^n}(m_{n-1,n})\right)	
		\]
		is composable. Thus, the sequence of elements in $M$ is $n$-composable and determines a unique morphism $\mu:\mathfrak{D}^k\to M$. It is not hard to check that this map determines a cocone, and thus, the proposition is proven. 
	\end{proof}
	
	\begin{corollary}
		For any weak partial monoid $(M,\cdot,1)$, the nerve $N(M)$ is a reduced spiny weakly 2-Segal simplicial set. 
	\end{corollary}
	
	\begin{proof}
		Since hom functors are continuous,  {Proposition} \ref{prop:nerve_W_2Seg} implies that $N(M)$ is weakly 2-Segal. It is immediate from the definitions that $M_0=\Hom(\{1\},M)$ is a singleton. Finally, the canonical map 
		\[
		\begin{tikzcd}[row sep=0em]
			M_k\arrow[r] & M_1\times \cdots \times M_1 \\
			(m_1,\ldots,m_k) \arrow[r,mapsto] & (m_1,\ldots,m_k)
		\end{tikzcd}
		\]
		is identified with the 1-Segal map of $N(M)_k$. Since this is manifestly injective, the corollary is proven. 
	\end{proof}
	
	\begin{proposition}\label{prop:Nerve_Ess_surj}
		Given $X\in \sf{W}2\sf{S}^{\on{sp},\ast}_\Delta$, there is a partial monoid $M$ such that $N(M)\cong X$. 
	\end{proposition}
	
	\begin{proof}
		Define $M:=X_1$, $1:=s_0(\ast)\in X_1$. Define the set of multaplicable pairs in $M$ to be the image of the 1-Segal map $X_2\to X_1\times X_1$, and define the multiplication to be $d_1:X_2\to X_1$. Note that, for $m\in M$, the elements $s_0(m)\in X_2$ and $s_1(m)\in X_2$ show that the pairs $(m,1)$ and $(1,m)$ are multiplicable, and both multiply to $m$.
		
		As in the 2-Segal case, we can use spiny-ness and reducedness to identify $X_n$ with a subset of $X_1^{\times n}$, and show that under this identification
		\[
		d_i(x_1,\ldots, x_n) =\begin{cases}
			(x_2,\ldots, x_n) & i=0 \\
			(x_1,\ldots, x_i\cdot x_{i+1},\ldots, x_n) & 0<i<n\\
			(x_1,\ldots,x_{n-1}) & i=n 
		\end{cases}
		\]
		and 
		\[
		s_i(x_1,\ldots,x_n)=(x_1,\ldots, x_{i},1,x_{i+1},\ldots, x_n). 
		\]
		
		The first weak 2-Segal condition requires that the diagram 
		\[
		\begin{tikzcd}[column sep=huge,row sep=large]
			X_3 \arrow[r]\arrow[d] & X_2\times^{\squareslashleft} X_2\arrow[d]\\
			X_2\times^{\squareslashright} X_2 \arrow[r] & X_1\times X_1\times X_1
		\end{tikzcd}
		\]
		be pullback, where the $\squareslashleft$ superscripts indicate which triangulation the 2-Segal pullback corresponds to. This immediately implies that if $(x\cdot y)\cdot z$ exists and $x\cdot (y\cdot z)$ exists, then the two are equal. Moreover, the canonical inclusions 
%		\[
%		X_n \to \lim_{\Delta^{\on{inj}}/\Delta^{\scr{T}}} X_i 
%		\]
\[
\begin{tikzcd}
X_n  \arrow[r] & \lim_{\Delta^{\on{inj}}/\Delta^{\scr{T}}} X_i   
\end{tikzcd}
\]
		for a triangulation $\scr{T}$ of $P_{n+1}$ show that the tuples representing elements of $X_n$ are $n$-composable with respect to this operation.
		
		Letting $M$ be the weak partial monoid so defined, this means there is a canonical inclusion of simplicial sets $f:X\to N(M)$ which is an isomorphism on sets of $0$, $1$, and $2$-simplices. 
		
		However, this morphism of simplicial sets induces canonical natural transformations of membrane space diagrams 
		\[
		\gamma_n:\on{MS}(X,n)\Rightarrow  \on{MS}(N(M),n). 
		\]
		The values of $\on{MS}(X,n)$ on the elements $\Delta^{\scr{T}}$ and $\on{Sp}^n$ are uniquely determined by the 2-truncation of $X$. As such, $\gamma_n$ is a natural isomorphism away from the cone point. Since the diagrams $\on{MS}(X,n)$ and $\on{MS}(N(M),n)$ are limit diagrams, this immediately shows that $\gamma_n$ is a natural isomorphism, and so $f$ is an isomorphism of simplicial sets, as desired.
	\end{proof}
	
	\begin{proof}[Proof (Of {Theorem}\ref{thm:nerve_characterization})]
		By {Proposition} \ref{prop:max_ass_data_FF} and {Proposition} \ref{prop:ess_im_magad}, the nerve is fully faithful. By {Proposition} \ref{prop:Nerve_Ess_surj}, it is essentially surjective onto $\sf{W}2\sf{S}^{\on{sp},\ast}_\Delta$. Thus, the theorem is proven. 
	\end{proof}

\subsection{Coskeletalness and partial magmas}

	It is 
	%already 
	{well-}known that 2-Segal simplicial sets 
	%must
	{are} 
	%be
	 3-coskeletal (see \cite[Corollary 1.7]{BOORS}). 
	 %However, as we will see, 
	{In this section we will see that} this result can be strengthened in the case of spiny weakly 2-Segal simplicial sets.

	For convenience, we will define the poset $\underline{\scr{I}}_n$ to be the full subposet of $\scr{I}_n$ on all of the objects \emph{except} $\Delta^n$, so that 
	\[
	\scr{I}_n\cong \underline{\scr{I}}_{n}\star [0]. 
	\]
	We view $\scr{I}_n$ and $\underline{\scr{I}}_n$ as subcategories of $\Set_\Delta$.  
	
	\begin{proposition}\label{prop:w2s_cosk}
		Let $X$ be a spiny, weakly 2-Segal set. Then, for any $n>2$, every diagram 
		\[
		\begin{tikzcd}[column sep=huge,row sep=large]
			\partial \Delta^n\arrow[d,hookrightarrow] \arrow[r,"f"] & X \\
			\Delta^n \arrow[ru,dashed,"\tilde f"']
		\end{tikzcd}
		\]
		admits a unique extension $\tilde f$. 
	\end{proposition}
	
	\begin{proof}
		In the case $n=3$, Take the pushout 
		\[
		\begin{tikzcd}[column sep=huge,row sep=large]
			\on{Sp}(\Delta^3)\arrow[r]\arrow[d] & \Delta^{\squaredivright}\arrow[d]\\
			\Delta^{\squarediv}\arrow[r] & \Delta^{w3}
		\end{tikzcd}
		\]
		Then consider the maps 
		\[
		\begin{tikzcd}
			\Hom(\Delta^{3},X)\arrow[r] & \Hom(\partial \Delta^3,X) \arrow[r] & \Hom(\Delta^{w3},X)
		\end{tikzcd}
		\]
		induced by the canonical maps $\Delta^{w3}\to \partial\Delta^{3}\to \Delta^3$. The composite is a bijection since $X$ is weakly 2-Segal. However, since the map $\Delta^{w3}\to \partial \Delta^3$ is bijective on simplices, the second map is an injection. This implies that the second map is a bijection, and so the first map must be as well, proving the case $n=3$. 
		
		Now suppose that $n>3$ and $X$ admits unique extensions to the extension problems above for $3\leq k<n$. Let $\Delta^{wn}$ denote the colimit of the diagram 
		\[
		\begin{tikzcd}
			\underline{\scr{I}}_n\arrow[r] & \Set_\Delta
		\end{tikzcd}
		\]  
		given by the inclusion. Note that $\partial\Delta^n$ and $\Delta^n$ equipped with the inclusions canonically form cones over this diagram, so that we obtain a sequence of maps of simplicial sets 
		\[
		\begin{tikzcd}
			\Delta^{wn} \arrow[r] & \partial \Delta^n \arrow[r] & \Delta^n. 
		\end{tikzcd}
		\]
		
		Now note that since $\partial \Delta^n$ is $(n-1)$-skeletal, and $X$ has the unique extension property against $\partial \Delta^k\to \Delta^n$ for $3\leq k<n$, a map $\partial \Delta^n \to X$ is uniquely determined by the underlying map $\on{sk}_2(\partial \Delta^n)\to X$. Since every 2-simplex of $\Delta^n$ is contained in at least one triangulation of $P_n$, the map 
		\[
		\begin{tikzcd}
			\Delta^{wn}\arrow[r] & \partial \Delta^n
		\end{tikzcd}
		\]
		is a surjection onto the 2-skeleton of $\partial \Delta^n$. Thus, the induced map  
		\[
		\begin{tikzcd}
			\Hom(\partial \Delta^3, X)\arrow[r] & \Hom(\Delta^{wn},X)
		\end{tikzcd}
		\]
		is injective. We can then consider the sequence of maps 
		\[
		\begin{tikzcd}
			\Hom(\Delta^{n},X)\arrow[r] & \Hom(\partial \Delta^n,X) \arrow[r] & \Hom(\Delta^{wn},X)
		\end{tikzcd}
		\]
		and note that, by the weak 2-Segal conditions, the composite is a bijection. This implies that the second map is surjective. Since the second map is also injective, it is a bijection. Thus the first map is a bijection as well, and so $X$ has the unique extension property against $\partial \Delta^n\to \Delta^n$. By induction, this completes the proof. 
	\end{proof}
	
	\begin{remark}
		{The converse of this statement does not hold.}	{T}here is an immediate counterexample. Consider the simplicial set $\Delta^{w3}$ constructed in the proof above. There is no non-degenerate higher simplex boundary contained in it, so it is 2-coskeletal, and it is clearly spiny. However, the identity map of $\Delta^{w3}$ does not extend to a map of the 3-simplex to $\Delta^{w3}$, and so $\Delta^{w3}$ is not weakly 2-Segal. 
	\end{remark}
	
	\begin{remark}
		It is, of course, well-known that there are numerous examples of 2-coskeletal simplicial sets which are not the nerves of categories. %However, t
		{T}his shows that every weakly associative partial monoid $M$  which is \emph{not} an actual monoid provides another such example. If the composition is {partially} defined, but there is only one ``object'', this cannot be the nerve of a category, however, the above shows that the nerve will necessarily be 2-coskeletal.
	\end{remark}
	
	\begin{corollary}\label{cor:lowest_w2s_sufficient_if_cosk}
		If $X$ is a 2-coskeletal, reduced, spiny simplicial set, then $X$ satisfies the weak 2-Segal conditions if and only if it satisfies the lowest-dimensional weak 2-Segal condition (i.e., at $n=3$). 
	\end{corollary}
	
	\begin{proof}
		Only one implication needs a proof. If $X$ is 2-coskeletal, reduced, spiny simplicial set satisfying the lowest-dimensional weak 2-Segal condition, then the construction of {Proposition} \ref{prop:Nerve_Ess_surj} gives us a weak partial monoid $M$ and an isomorphism between the 2-truncations of $X$ and $N(M)$. Since both simplicial sets are 2-coskeletal, it follows that $X\cong N(M)$, completing the proof.
	\end{proof}

	We now prove a yet stronger statement.
	
	\begin{proposition}
		The essential image of the functor 
		\[
		\begin{tikzcd}
			N: &[-3em] \sf{Mag} \arrow[r] & \Set_\Delta 
		\end{tikzcd}
		\]
		consists of precisely the spiny, reduced, 2-coskeletal simplicial sets. 
	\end{proposition}

	\begin{proof}
		First, suppose that $(M,\cdot,1)$ is a partial unital magma. By {Proposition} \ref{prop:ess_im_magad}, $X:=N(M)$ is spiny and reduced, so in the commutative diagram 
		\[
		\begin{tikzcd}[column sep=huge,row sep=large]
			\Set_\Delta(\partial \Delta^n,X)\arrow[d,hookrightarrow] & X_n \arrow[l]\arrow[dl,hookrightarrow]\\
			X_1^{\times n}
		\end{tikzcd}
		\]
		induced by the inclusions 
		\[
		\begin{tikzcd}[column sep=huge,row sep=large]
			\partial \Delta^n \arrow[r] & \Delta^n \\
			\on{Sp}(\Delta^n)\arrow[u]\arrow[ur] 
		\end{tikzcd}
		\]
		the two vertical maps are injective. As such, the horizontal map is as well. We thus need only prove that the horizontal map is surjective. 
		
		Let $(m_1,\ldots,m_n)$ be the spine of an $n$-simplex boundary in $X$ for $n>2$. For any triangulation $\scr{T}$ of $P_{n+1}$, $\Delta^{\scr{T}}\subset \partial \Delta^n$, and thus, the fact that $(m_1,\ldots,m_n)$ forms the spine of an $n$-simplex boundary implies that $(m_1,\ldots,m_n)$ is $T$-composable for the binary tree $T$ dual to $\scr{T}$. Thus, we see that $(m_1,\ldots,m_n)$ is $n$-composable. Since the terminal edge of $\Delta^n$ is contained in $\partial \Delta^n$, all of the ways of multiplying the tuple $(m_1,\ldots,m_n)$ coincide, and so this tuple is associable. 
		
		Moreover, since the $0$- and $n$-faces of this tuple necessary correspond to fully associable tuples by {Lemma} \ref{lem:Nerve_n_composable}, it follows that $(m_1,\ldots,m_n)$ is, in fact, fully associable, and thus defines an $n$-simplex in $X=N(M)$. 
		
		The other implication follows immediately from the proof of {Proposition} \ref{prop:ess_im_magad}, and the proposition is proven. 
	\end{proof}

	\begin{remark}
		The characterization of partial unital magmas as spiny, reduced, 2-coskeletal simplicial sets provides another way of understanding the issues discussed in {Remark} \ref{rmk:Chermak_non-equivalent_same}. Effectively, this result means that a maximal set of associativity data can be uniquely assigned to any partial unital magma. That is, working in $\sf{Mag}^{\on{ad}}$ means ignoring that some fully associative tuples are associative, leading to non-isomorphic encodings of the same underlying algebraic structure. 
		
	%	One might point out that many examples of partially-defined multiplications (e.g., the commutative nerves defined in the next section) are also defined by simply forgetting some of an algebraic structure. The key difference is that a morphism between, e.g., commutative nerves of groups need not define a morphism between the groups themselves, whereas any morphism between partial monoids with associativity data necessarily defines a morphism between partial monoids with the maximal associativity data.   
	\end{remark}

\section{{Invertibility and (weak) partial groups}}
\label{sec:invertibility}

{In the next section, we will see that invertibility plays a crucial role in our definition of simplicial effects.}
Accordingly, we now digress into the notion of invertibility for weak partial monoids, which necessarily brings us into contact with Chermak's notion of a partial group, as well as the corresponding concept for weak partial monoids.
These objects form the opposite extremal case from effect algebras, yet the analysis of inverses and invertibility is rather subtle.

\begin{definition}
	Let $(M,\cdot,1)$ be a weak partial monoid, and let $m\in M$. We say that an element $n\in M$ is 
	\begin{itemize}
		\item a \emph{left inverse} to $m$ if $(n,m)$ is multiplicable and $n\cdot m=1$,
		\item a \emph{right inverse} to $m$ if $m$ is a left inverse to $n$, and
		\item an \emph{inverse} (or a \emph{2-sided inverse}) to $m$ if it is both a left and right inverse to $m$.  
	\end{itemize}
	We will call $M$ a \emph{weak partial monoid with inverses} if every element $m\in M$ has a 2-sided inverse, which we denote by $m^{-1}$. The category $\sf{WPM}^{\on{inv}}$ of weakly associative partial monoids with inverses is then a full subcategory of $\sf{WPM}$. 
\end{definition}
	
\begin{lemma}\label{lem:inversion_props}
	Let $(G,\cdot,1)$ and $(H,\ast,e)$ be two weakly partial monoids with inverses. 
	\begin{enumerate}
		\item Let $g\in G$, and $u$ be a left or right inverse to $g$. Then $u=g^{-1}$. 
		\item For $g\in G$, $(g^{-1})^{-1}=g$. 
		\item Let $f:G\to H$ be a morphism of weak partial monoids. Then for any $g\in G$, $f(g^{-1})=f(g)^{-1}$.
	\end{enumerate}
\end{lemma}

\begin{proof}
	For part (1), we treat only the case where $u$ is a left inverse, the other case is similar. Note that $(u\cdot g)\cdot g^{-1}$ and $u\cdot(g\cdot g^{-1})$ are multiplicable, and so by weak associativity, they are equal. Thus
	\[
	g^{-1}=1\cdot g^{-1}=(u\cdot g)\cdot g^{-1}=u\cdot(g\cdot g^{-1})= u\cdot 1= u.
	\] 
	Claim (2) is an immediate consequence of (1). 
	
	For (3), note that,
	\[
	e=f(1)=f(g^{-1}\cdot g) =f(g^{-1})\cdot f(g)
	\]
	so that $f(g^{-1})$ is left inverse to $f(g)$. Then apply part (1). 
\end{proof}

{One might be tempted to define partial groups to simply be weak partial monoids with inverses, however, this is insufficient to guarantee that, e.g., $(g\cdot h)^{-1}=(h^{-1}\cdot g^{-1})$, since we cannot use weak associativity to guarantee that the latter pair is multiplicable. We thus impose further conditions in our definition of weakly associative partial groups. }

\begin{definition}
	Let $(M,\cdot,1)$ be a weak partial monoid. The \emph{opposite weak partial monoid $M^\op$} has the same underlying set and unit as $M$, but multiplication given by $m\ast n=n\cdot m$. 
	
	We call a weak partial monoid $(G,\cdot,1)$ with inverses \emph{involutive} if $(-)^{-1}:G\to G^\op$ is an isomorphism of weak partial monoids. We call $G$ a \emph{weakly associative partial group} if for every fully associable tuple $(g_1,\ldots, g_n)$, the tuple $(g_1,\ldots, g_n,g_n^{-1},\ldots, g_1^{-1})$ is fully associable.   
\end{definition}

\begin{lemma}
	Let $(G,\cdot,1)$ be a weak partial monoid with inverses. 
	\begin{enumerate}
		\item $G$ is involutive if and only if $(-)^{-1}:G\to G^\op$ extends to an isomorphism of simplicial sets 
		\[
		\begin{tikzcd}
			(-)^{-1}:&[-3em] N(G) \arrow[r] & N(G)^\op. 
		\end{tikzcd}
		\]
		\item If $G$ is a weakly associative partial group, then $G$ is involutive. 
	\end{enumerate}
\end{lemma}

\begin{proof}
	For (1), it is easy to see that $N(G^\op)\cong N(G)^\op$. Since $N(G)$ is spiny, if there exists such an extension, it is given on $n$-simplices by 
	\[
	\begin{tikzcd}
		(g_1,\ldots,g_n)\arrow[r,mapsto] & (g_n,\ldots,g_1). 
	\end{tikzcd}
	\]
	Since the nerve is fully faithful, statement (1) follows. 
	
	For (2), given any fully associable tuple $(g_1,\ldots,g_n)$ in $G$, the fact that $(g_1,\ldots, g_n,g_n^{-1},\ldots, g_1^{-1})$  is fully associable means that the subsequence $(g_n^{-1},\ldots, g_1^{-1})$ is as well. Thus, $(-)^{-1}$ is a partial monoid homomorphism. 
\end{proof}

We then compare our notion of weakly associative partial group to the usual definition of partial groups \cite[Definition 2.1]{Chermak}.

\begin{definition}\label{defn:chermak_PG}
	Let $(M, \mathbf{D},\Pi)$ be a partial associativity structure. An \emph{inversion} on $M$ is an involution 
	\[
		\begin{tikzcd}
			(-)^{-1}:&[-3em] M\arrow[r] & M .
		\end{tikzcd}
	\] 
	We abusively also write $(-)^{-1}$ for the involution 
	\[
	\begin{tikzcd}[row sep=0em]
		\mathbf{W}(M) \arrow[r] &\mathbf{W}(M)\\
		(g_1,\ldots,g_n) \arrow[r,mapsto] & (g_n^{-1},\ldots,g_1^{-1}).
	\end{tikzcd}
	\]
	A \emph{partial group} is a partial associativity structure equipped with an inversion $(-)^{-1}$ such that 
	\begin{itemize}
		\item[5.] For any $u\in \bf{D}$, $u\circ u^{-1}\in \bf{D}$ and $\Pi(u\circ u^{-1})=1$.   
	\end{itemize} 
\end{definition}

The following proposition justifies the terminology ``weakly associative partial group''. 

\begin{proposition}
	Let $(M,\cdot, 1)$ be a weak partial monoid. Then the maximal associativity structure on $M$ defines a partial group if and only if $M$ is a weakly associative partial group. 
\end{proposition}

\begin{proof}
	This is almost immediate from the definitions. The only possible issue is the possibility that a PAS could be equipped with multiple different inversions satisfying condition (5) from {Definition} \ref{defn:chermak_PG}. However, by {Lemma}\ref{lem:inversion_props} part (1), this is not the case, completing the proof. 
\end{proof}

{Commutative nerves are our key examples.}
	
	\begin{corollary}\label{cor: comm nerve weak partial group}
		For a group $G$, $N(\ZZ,G)$ and $N(\ZZ_{/d},G)$ are weakly associative partial 
		groups.
		%monoids. 
	\end{corollary}
	
	\begin{proof} 
We will again write the proof only for \( N(\mathbb{Z}, G) \). By {Lemma} \ref{lem:comm_nerve_2cosk} and {Corollary} \ref{cor:lowest_w2s_sufficient_if_cosk}, it suffices to show that \( N(\mathbb{Z}, G) \) satisfies the lowest weak 2-Segal condition. This is identical to the \( n = 3 \) case in the proof of {Lemma} \ref{lem:comm_nerve_2cosk}. 
{Therefore,} \( N(\mathbb{Z}, G) \) is a weak partial monoid. Since it also serves as an example of a partial group in the sense of {Definition} \ref{defn:chermak_PG}, it follows immediately that it is a weakly associative partial group.
	
	\end{proof}

%However, i
{It is also} worth noting that there are canonical constructions of partial groups which are \emph{not} weakly associative. 

\begin{example}
	This is a general class of (non-)examples. Let $G$ be a group, $Z$ a set with a $G$-action, and $Y\subset Z$. Define a reduced simplicial set $L_Y(G)$ whose set of $n$-simplices is the set 
	\[
	L_Y(G)_n:=\left\lbrace (g_1,\ldots,g_n)\middle| \exists \{y_i\in Y\}_{i=0}^n \text{ s.t. } g_iy_{i-1}=y_i\;\; \forall 0\leq i\leq n \right\rbrace.
	\]
	It is obvious that this is spiny, and one can easily check that it is a partial group in the sense of {Definition} \ref{defn:chermak_PG}.  
	
	%However, t
	{T}here are examples which are not weakly associative. Consider, for instance, the group $\ZZ_{/4}$ acting on $\ZZ_{/4}$, equipped with the subset $Y$ consisting of $0$, $1$, and $2$. It is immediate that the pairs $(1,1)$, $(2,1)$ and $(1,2)$ are multiplicable in $L_Y(\ZZ_{/4})$. However, the tuple $(1,1,1)$ is not multiplicable, and so $L_Y(\ZZ_{/4})$ is not weakly associative.

	It is
	%, however, 
	important to point out that considering these examples as partial groups is somewhat artificial, as they are more naturally categories. Given $G$, $Z$, and $Y\subset Z$ as above, we can form a groupoid $C_Y(G)$, which is the full sub-groupoid of the action groupoid $Z_{//G}$ on the objects in $Y$. The nerve of $C_Y(G)$ is a simplicial set, and $L_Y(G)$ is simply the image of this nerve in the nerve of $G$. 
\end{example}

\section{Simplicial effects}
\label{sec:Simplicial effects}

{
Effect algebras and effect algebroids rely on partial operations that are associative in a stronger sense than is satisfied by the canonical examples of simplicial measurements, which also behave like effects in the simplicial framework of \cite{sdist}. Motivated by our weaker notion of associativity---which holds for those simplicial examples---we introduce a category of \emph{simplicial effects}. For our definition, we begin by identifying the key ideas underlying effect algebras and their generalization to effect algebroids.
}

\subsection{The orthocomplement}\label{sec:orthocomplement}

{  One of the first structures we analyze in the context of weakly $2$-Segal cyclic sets is the orthocomplement, 
  a fundamental concept in effect algebras and effect algebroids. 
  In this setting, we treat spiny, weakly $2$-Segal simplicial sets as multivalued categories, 
  following \cite[\S 3.3]{DK}.}
%{L}et us fix notation
% and interpretation for a spiny, weakly 2-Segal simplicial set $X$  
%We will think of $X$ 
%as a weak version of a \emph{$\mu$-category}, in the vein of \cite[\S 3.3]{DK}.  
\begin{itemize}
	\item We will refer to the elements of $X_0$ as the \emph{objects} of $X$.
	%, and denote them by Latin minuscules, $a,b,c,x,y,z,\ldots$. 
	\item For objects $x,y\in X_0$, 
	%we will sometimes denote the fibre of the map 
	{an element of the fibre of the map}
	\[
	\begin{tikzcd}
		X_1 \arrow[r,"{(d_1,d_0)}"] & X_0\times X_0  
	\end{tikzcd}
	\]
%	over $(x,y)$ by $\Hom_X(x,y)$, 
{will be denoted by $f:x\to y$,}
and {we will} call the elements of $X_1$ \emph{morphisms}. 
	\item We will use the composition symbol $\circ$ to denote the partially defined map specified by the left-injective span 
	\[
	\begin{tikzcd}
		X_1\times_{X_0} X_1 & X_2\arrow[l,hookrightarrow,"{(d_2,d_0)}"']\arrow[r,"d_1"] & X_1 
	\end{tikzcd}
	\]
	and will explicitly parenthesize multiple compositions. 
	\item We will denote the identity morphism (degenerate 1-simplex) on $x$ by $0_x$. 
\end{itemize} 
Notice that, for $f:x\to y$, $f$ can be composed with $0_x$ and $0_y$, and  $0_x\circ f=f=f\circ 0_y$.
%If we are 
{G}iven a spiny, weakly 2-Segal cyclic set $X$, we also
%will 
make the following definitions: 
\begin{itemize}
	\item We define the orthocomplement $(-)^\perp$ to be the map $\tau_1:X_1\to X_1$. 
	\item We define the element $1_x$ to be $(0_x)^\perp$. 
\end{itemize}

\begin{lemma}\label{lem:conds_on_swset}
	For a spiny, weakly 2-Segal cyclic set, the following identities hold. 
	\begin{enumerate}
		\item For $f:y\to z$ and $g:x\to y$, we have that
		\[
		f\circ g=h \Longleftrightarrow g\circ h^\perp=f^\perp
		\]
		in the additional sense that if one composite is defined, so is the other.
		\item For $f\in X_1$, $(f^\perp)^\perp=f$. 
		\item For $x\in X_0$, $(1_x)^\perp=0_x$.
		\item For any $f:y\to x$ and $g:x\to y$, if $f\circ g=1_x$, then $f=g^\perp$. 
	\end{enumerate}
\end{lemma}
\begin{proof}
	Part (2) is immediate from the fact that $\tau$ is an involution, and part (3) is a special case of (2). 
	
	To see (1), let $\alpha\in X_2$ be the witness to the composition $f\circ g=h$, i.e., $d_0(\alpha)=f$, $d_2(\alpha)=g$, and $d_1(\alpha)=h$. Then applying the automorphism $\tau_2$ to $\alpha$ yields an element $\beta=\tau_2(\alpha)$ with 
	\[
	d_0(\beta)= d_0(\tau_2(\alpha))=d_2(\alpha)=g
	\]
	\[
	d_2(\beta)=d_2(\tau_2(\alpha))=\tau_1(d_1\alpha)=\tau_1(h)=h^\perp 
	\]
	and 
	\[
	d_1(\beta)=d_1(\tau_2(\alpha))=\tau_1(d_0(\alpha))=\tau_1(f)=f^\perp 
	\]
	so that $g\circ h^\perp=f^\perp$. The reverse implication follows by applying $\tau_2^{-1}=\tau_2^2$. 
	
	Finally, to see part (4), we can apply (1) to the equation $f\circ g=1_x$ to get $g\circ 1_x^\perp=f^\perp$. Since $1_x^\perp=0_x$ is the identify, this means $g=f^\perp$, or, equivalently, $f=g^\perp$ as desired. 
\end{proof}

\begin{remark}
	Note that one upshot of the discussion above is that, allowing ``associativity'' to mean ``weak associativity'', our spiny weakly 2-Segal sets already satisfy the first three conditions in the definition of effect algebroids, {Definition \ref{def:effect algebroid}.}
	%\cite[Defn 4.1.2]{Roumen}.
\end{remark}

\subsection{The zero-in-one law}
\label{sec:zero one law}

{The next property we consider is the zero-in-one-law. 
  From Lemma~\ref{lem:conds_on_swset}, we see that this condition 
  is equivalent to the absence of inverses. 
  Let us recall the zero-in-one-law from 
  Definition~\ref{def:effect algebroid}:}
\begin{center}
	\emph{For $f:x\to y$, if $1_y\circ f$ or $f\circ 1_x$ is defined, then $f=0_x$ (and $x=y$).}
\end{center}
Applying condition (1) of Lemma \ref{lem:conds_on_swset} twice, we find that
%, e.g.
\[
f\circ 1_x =h \Longleftrightarrow 1_x\circ h^\perp=f^\perp \Longleftrightarrow h^\perp\circ f =0_x. 
\] 
That is, $f$ is composable with $1_x$ if and only if $f$ has a left inverse. Similarly, 
\[
1_y\circ f=h \Longleftrightarrow f\circ h^\perp=0_x.
\]
That is, $1_y$ is composable with $f$ if and only if $f$ has a right inverse. Thus, the zero-in-one law becomes 
\begin{center}
	\emph{If $f$ has a left or a right inverse, then $f$ is an identity.}
\end{center}
These conditions are particularly appealing, since they can be written without reference to the cyclic structure, that is, using only morphisms in the simplex category. The set of pairs of morphisms $(f,g)$ with $f$ left inverse to $g$ and $g$ right inverse to $f$ is the pullback of the diagram
\[
\begin{tikzcd}[column sep=huge,row sep=large]
	& X_2 \arrow[d,"d_1"]\\
	X_0\arrow[r,"s_0"'] & X_1
\end{tikzcd}
\]
The zero-in-one law is thus equivalent to requiring that the space of such pairs is precisely the space of pairs $(0_x,0_x)$ of identities on some object $X$. Thus, the zero-in-one-law is equivalent to requiring that the diagram 
\begin{equation}\label{eq:Z-I-O}
	\begin{tikzcd}[column sep=huge,row sep=large]
		X_0 \arrow[r,"s_0\circ s_0"]\arrow[d,equal] & X_2 \arrow[d,"d_1"]\\
		X_0\arrow[r,"s_0"'] & X_1 
	\end{tikzcd}
\end{equation}
is a pullback. {This is precisely the condition implementing the zero-in-one law in Theorem \ref{thm:characterization fo EAd}.}
%Notice that (up to labeling the circles) this square is precisely that dual to the square (Z) in \cite[Lemma 5.1.2 \& Theorem 5.1.4]{Roumen}. Thus, as in \cite{Roumen}, we can conclude that a spiny, weakly 2-Segal cyclic set has an orthocomplement satisfying the zero-in-one law if and only if it sends this diagram to a pullback.

%\subsection{Inverses, simplicial effects, and partial groups}
 
%\comm{==separated==}	
%	On the other hand,
{In particular, in line with our invertibility discussion in Section \ref{sec:invertibility},} we will call $M$ \emph{inverseless} if $m\in M$ having a left or right inverse implies that $m=1$.   
%\end{definition}
We can thus reformulate the preceding 
%section 
{paragraph}
in the following form.

\begin{lemma}
	For a partial monoid $(M,\cdot,1)$, the following are equivalent.
	\begin{enumerate}
		\item $M$ is inverseless. 
		\item The square (\ref{eq:Z-I-O}) is {a} pullback.
		% in $N(M)$. 
		\item The zero-in-one law holds in $M$. 
	\end{enumerate}
\end{lemma}

{Generalizing this observation into a definition will be important for the definition of simplicial effects in the next section.}

\begin{definition}
%	We will say a simplicial set $X$ is \emph{inverseless} if the square (\ref{eq:Z-I-O}) is pullback.  
{A simplicial set $X$ is said to be \emph{inverseless} if the square (\ref{eq:Z-I-O}) is a pullback.}
\end{definition}

%With this definition in hand, we can now define our simplicial effects. 

\subsection{{The category of} simplicial effects}

{Now we are ready to state our main definition: a simplicial version of \emph{effect} that generalizes both effect algebras and effect algebroids. We will make use of the weak version of associativity introduced in Section~\ref{sec:weak associativity}, which is satisfied by our canonical example of simplicial measurements \(\displaystyle P_\hH(N\ZZ_{/d})\). We also rely on our reinterpretation of the two key properties of effect algebroids---the orthocomplement and the zero-one-law.}

\begin{definition}\label{def:simplicial effect}
	A \emph{simplicial effect} is a spiny, inverseless, weakly 2-Segal cyclic set.  
{The category of simplicial effects $\SimpEff$ is the full subcategory of $\Set_\Lambda$ consisting of simplicial effects as its objects.} 	
\end{definition}

Note that since 2-Segal cyclic sets are weakly 2-Segal, this definition subsumes Roumen's characterization of the simplicial sets corresponding to effect algebras {in Theorem \ref{thm:characterization fo EAd}.}  
%{Moreover, if a simplicial set $X$ satisfies the 2-Segal conditions, then condition (U) mentioned in the theorem, being a sub-pullback, is equivalent to the requirement that the spine map for the 2-simplex is injective. Consequently, the simplicial set $X$ is also spiny.} 
%\wscomm{This is not quite right as stated. My suggested rewrite follows. } 
{Moreover, condition (U) from Theorem \ref{thm:characterization fo EAd} can equivalently be stated as requiring that the spine map associated to the 2-simplex is injective. Once the 2-Segal conditions are also imposed, this implies that the simplicial set is spiny.}

{Our next task is to provide a simplicial effect that is not an effect algebroid.} 
{The starting point is the construction given in Definition \ref{def:EX}. First, we observe that there is a case where the resulting construction is a simplicial effect.}

\begin{proposition}\label{pro:ES1}
There is an isomorphism of simplicial sets
%\[
%E(S^1) \to N(E).
%\]
\[
\begin{tikzcd}
E(S^1) \arrow[r] & N(E) .
\end{tikzcd}
\]
%where $ N_{\leq 1}(R)$ is the simplicial subset of the nerve space $N(R)$ of $(R,+)$, consisting of tuples whose sum is at most $1$. 
%\is{We need order to say this; it can be the one we discussed recently, that $a\leq b \Leftrightarrow \exists c: a+c=b$, but for this we need a zero-sum-free semiring.}
\end{proposition}
\begin{proof}
This follows from the following description of the simplicial circle.
The set of $n$-simplices of $S^1$ is given by
$$
(S^1)_n = \left\lbrace
\begin{array}{ll}
\set{\star} & n=0 \\
\set{\star, \theta^1,\cdots,\theta^n} & n\geq 1.
\end{array}
\right.  
$$ 
The simplicial structure maps on $\theta^i$ are given by
$$
d_j(\theta^i) = \left\lbrace
\begin{array}{ll}
\theta^{i-1} & j< i{\text{ and }}  1<i \\ 
\theta^{i} & i\leq j {\text{ and }}  i<n \\
\star & \text{otherwise}
\end{array}
\right.
\;\;\;\text{ and }
\;\;\;
s_j(\theta^i) = \left\lbrace
\begin{array}{ll}
\theta^{i+1} & j< i \\ 
\theta^{i} & i\leq j.\\
\end{array}
\right.
$$ 
Any simplicial structure map sends $\star$ to itself. See, e.g., \cite[\S 2.5]{okay2023equivariant}.

For elements $a.b\in E$, we write $a\leq b$ if there exists $c\in E$ such that $(a,c)\in E_2$ and $b=a+c$.
An $n$-simplex of $E(S^1)_n$ is specified by a tuple $(a_1,a_2,\cdots,a_{n})$ where $a_i\in E$ and $\sum_i a_i \leq 1$. The value assigned to $\theta^i$ is given by $a_{i}$ and the value assigned to $\star$ is determined by this tuple, which is given by $1-\sum_i a_i$. The simplicial structure of $S^1$  can be used to show that
$$
d_j(a_1,a_2,\cdots,a_{n}) = \left\lbrace
\begin{array}{ll}
(a_2,\cdots,a_{n}) & j=0\\ 
(a_2,\cdots,a_j+a_{j+1},\cdots, a_{n}) & 0<j<n \\
(a_2,\cdots,a_{n-1}) & j=n
\end{array}
\right.
$$
and
$$
s_j(a_1,a_2,\cdots,a_{n}) = (a_1,\cdots,a_j,0,a_{j+1},\cdots,a_n).
$$ 
\end{proof}

{Since nerves of effect algebras are simplicial effects, $E(S^1)$ is a simplicial effect.
However, our motivating example $P_{\hH}(N\ZZ_{/d})$, defined as $E(N\ZZ_{/d})$ with $E=\Proj(\hH)$, lies in the opposite regime. By Proposition~\ref{prop:U(H)_P(H)} and Corollary~\ref{cor: comm nerve weak partial group}, this simplicial set is a weakly associative partial group. Thus, in contrast to simplicial effects, inverses are present.}
To construct our key example, we will address this issue with the following construction.

\begin{definition}
	Let $Y$ be a simplicial set, and $X\subset \on{tr}_2(Y)$ a simplicial subset of its 2-truncation. Define a simplicial subset $F_2^Y(X)\subset Y$ as follows. A simplex
	\[
	\begin{tikzcd}
		\sigma: &[-3em] \Delta^n \arrow[r] & Y 
	\end{tikzcd}
	\]  
	factors through $F_2^Y(X)$ if and only if the 2-truncation 
	\[
	\begin{tikzcd}
		\on{tr}_2(\sigma):&[-3em] \on{tr}_2(\Delta^n)\arrow[r] & \on{tr}_2(Y)
	\end{tikzcd}
	\]
	factors through $X$. We call $F_2^Y(X)$ the \emph{full simplicial subset of $Y$ on $X$.} 
\end{definition}

{Next, we observe that this construction preserves the properties of being spiny, weak \( 2 \)-Segal, and cyclic structures.
}

\begin{lemma}\label{lem:full_spiny}
	If $Y$ is spiny, then $F_2^Y(X)$ is spiny. 
\end{lemma}

\begin{proof}
	Immediate from the commutativity of 
	\[
	\begin{tikzcd}[column sep=huge,row sep=large]
		X_1\times_{X_0}\cdots\times_{X_0}X_1 \arrow[d,hookrightarrow] & F_2^Y(X)_n\arrow[l]\arrow[d,hookrightarrow]\\
		Y_1\times_{Y_0}\cdots\times_{Y_0}Y_1   & Y_n\arrow[l,hookrightarrow]
	\end{tikzcd}
	\] 
\end{proof}

\begin{lemma}\label{lem:full_w2s}
	If $Y$ is weak 2-Segal, then $F_2^Y(X)$ is weak 2-Segal. 
\end{lemma}

\begin{proof}
	Recall the simplicial sets $\Delta^{wn}$ for $n\geq 3$ constructed in the proof of {Proposition} \ref{prop:w2s_cosk}. Being weakly 2-Segal amounts to having the unique right extension property against 
	\[
	\begin{tikzcd}
		\Delta^{wn} \arrow[r] &\Delta^n. 
	\end{tikzcd}
	\]
	However, this map factors through the boundary and is surjective on 2-truncations, and so the lemma follows. 
\end{proof}

\begin{lemma}\label{lem:full_cyclic}
	Suppose that $X$ is a cyclic subset of the 2-truncation of a cyclic set $Y$. Then $F_2^Y(X)$ is a cyclic subset of $Y$. 
\end{lemma}

\begin{proof}
	From the cyclic identities, it follows that if $\sigma$ lies in $F_2^Y(X)$, then every 2-dimensional face of $\tau_n\sigma$ is $\tau_2^k$ of a 2-simplex in $X$ for some $k$. Thus, the 2-truncation of $\tau_n \sigma$ factors through $X$, as desired.  
\end{proof}

 	Now, {we are ready to introduce our example.}

\begin{construction}\label{const:key example}
 Let $\scr{H}=\CC^3\otimes\CC^3$, and set $Y=P_{\scr{H}}(N\ZZ_{/3})$. Note that this is weakly 2-Segal, spiny, and reduced by {Proposition} \ref{prop:U(H)_P(H)}. We can write an $n$-simplex $\Pi$ of $Y$ as an indexed set of projectors $\{\Pi^{a_1,\ldots,a_n}\}_{a_i\in \ZZ_{/3}}$. 
	Define $X\subset \on{tr}_2(Y)$ by 
	\[
	\begin{aligned}
		X_0&=Y_0\\
		X_1&=Y_1\\
		X_2&=\{\{\Pi^{ab}\}_{a,b\in\ZZ_{/3}}\mid \Pi^{11}=\Pi^{21}=\Pi^{12}=0\}.
	\end{aligned}
	\] 
Set $Z=F_2^Y(X)$. 
\end{construction}

{
The construction is such that the unwanted property of having inverses is eliminated while the weak $2$-Segal property is preserved. In addition, the resulting simplicial set is not $2$-Segal, hence lies outside Roumen's category of effect algebroids.}

%CO:added braces in the code below
\begin{figure}[htb]
	\begin{center}
		 \includegraphics[width=0.9\textwidth]{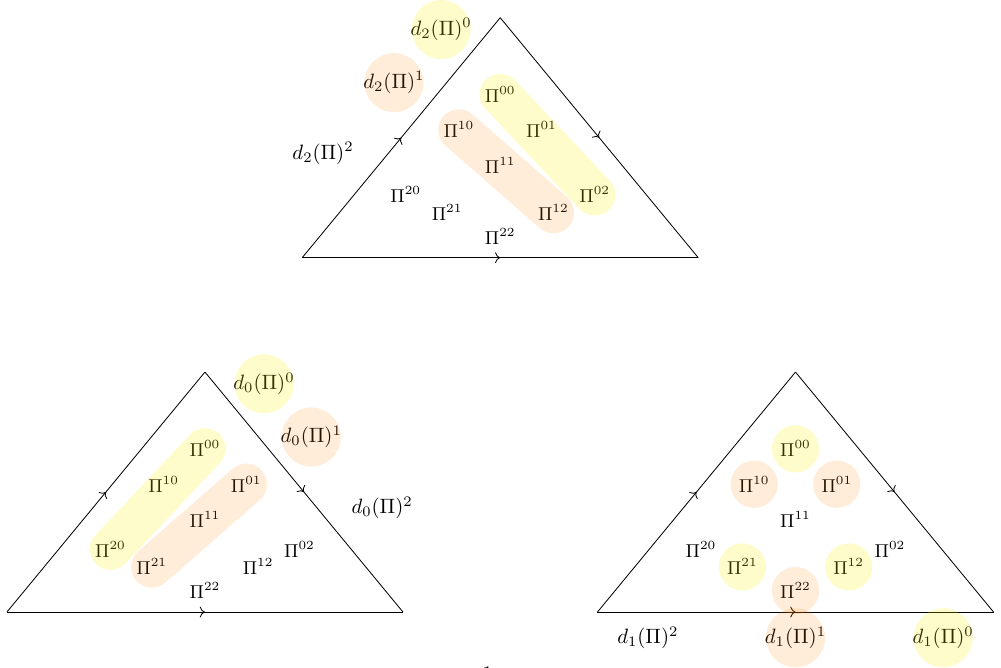}
	\end{center}
	\caption{The three face maps of a 2-simplex in $P_{\scr{H}}(N(\ZZ/3))$. Projectors inside the simplex highlighted in the same color are summed to yield the projector outside the simplex highlighted in that color. 
	%\comm{it would be perfect to align the red lines with the blue and yellow ones.}
	}
\end{figure}

\begin{proposition}
Let $Z$ denote the simplicial effect defined in  Construction \ref{const:key example}. Then
\begin{enumerate}
\item $Z$ is a simplicial effect (Definition \ref{def:simplicial effect}).
\item $Z$ is not an effect algebroid (Definition \ref{def:effect algebroid}).
% inverseless
\end{enumerate}
%We now aim to show two further properties: (1) $Z$ is inverseless, and (2)  $Z$ is \emph{not} $2$-Segal.  
\end{proposition}

\begin{proof}
By definition, $Z$ is reduced, by {Lemma} \ref{lem:full_spiny}, $Z$ is spiny, and by {Lemma} \ref{lem:full_w2s}, $Z$ is weak 2-Segal. 
It follows from {Lemma} \ref{lem:full_cyclic} that $Z$ is cyclic, since the conditions defining $X_2$ are preserved by the cyclic automorphisms.  
	To see that $Z$ is inverseless, we consider a 2-simplex $\Pi$ representing an element in the pullback of diagram (\ref{eq:Z-I-O}). This means that $d_1(\Pi)$ is the degenerate 1-simplex given by 
	\[
	d_1(\Pi)^0=\mathbb{1},\;\; d_1(\Pi)^1=\mathbb{0},\;\; d_1(\Pi)^2=\mathbb{0}. 
	\]
	Since, by construction, $\Pi^{11}=\Pi^{21}=\Pi^{12}=0$, the definition of the face maps yields the relations 
	\[
	\begin{aligned}
		\mathbb{1} &= \Pi^{00} \\
		\mathbb{0} &= \Pi^{22}+\Pi^{10}+\Pi^{01}\\
		\mathbb{0} &= \Pi^{20}+\Pi^{02}
	\end{aligned}
	\]
	and thus, by normalization, 
	\[
	\Pi^{ij}=\begin{cases}
		\mathbb{1} & ij=00\\
		\mathbb{0} & \text{else}.  
	\end{cases}
	\]
	This implies that $\Pi$ is a totally degenerate 2-simplex, and so diagram (\ref{eq:Z-I-O}) is pullback, as desired.

	Finally, we aim to show that $Z$ is \emph{not} 2-Segal. We first fix some notation. Let $\{|a\rangle\}_{a\in{0,1,2}}$ be a basis for $\CC^3$, and write $|ab\rangle=|a\rangle\otimes|b\rangle$ for the corresponding basis of $\CC^3\otimes\CC^3$. Write $\Gamma^{ab}:=|ab\rangle\langle ab|$ for the projector which projects onto the basis vector $|ab\rangle$. We then define a 2-simplex  $\Pi$ of $Z$ by setting 
	\[
	\Pi^{ab}:= \begin{cases}
		\mathbb{0} & ab=11,21,12\\
		\Gamma^{01}+\Gamma^{11}+\Gamma^{21}+\Gamma^{12} & ab=01\\
		\Gamma^{ab} & \text{else}.
	\end{cases}
	\]
	Since the $|ab\rangle$ form a basis of $\scr{H}$,it follows immediately that normalization holds.  
	
	Now, define vectors 
	\[
	\begin{aligned}
		|+\rangle & =\frac{1}{\sqrt{2}}\left(|02\rangle +| 20\rangle \right)\\
		|-\rangle & =\frac{1}{\sqrt{2}}\left(|02\rangle - 20\rangle \right)\\
	\end{aligned}
	\] 
	and write $\Gamma^+$  and $\Gamma^-$ for the corresponding projectors.  
	We then define a 2-simplex $\Psi$ of $Z$ given by
	\[
	\Psi^{ab}:= \begin{cases}
		\Gamma^+ & ab=20\\
		\Pi^{01} & ab=10 \\
		\Gamma^{00} & ab=01\\
		\Gamma^{-} & ab=22\\
		\mathbb{0}& \text{else}.
	\end{cases}
	\]
	It is a quick computation to see that $d_2(\Psi)=d_1(\Pi)$, and so these two 2-simplices form a map $\Delta^{\squarediv}\to Z$. However, setting $\omega= e^{2\pi i/3}$. The unitaries $A$, $B$, and $C$ corresponding to $d_2(\Pi)$, $d_0(\Pi)$ and $d_0(\Psi)$ under the isomorphism of {Proposition} \ref{prop:U(H)_P(H)} are 
	\[
	\begin{aligned}
		A & = (\Gamma^{20}+\Gamma^{22})+\omega \Gamma^{10}+\omega^2 (\Pi^{01}+\Gamma^{00}+\Gamma^{02})\\ 
		B & = (\Gamma^{00}+\Gamma^{10}+\Gamma^{20})+\omega \Pi^{01} + \omega^2(\Gamma^{02}+\Gamma^{22}) \\
		C &= (\Pi^{01}+\Gamma^{10}+\Gamma^{22}+\Gamma^+) +\omega \Gamma^{00} +\omega^2\Gamma^{-}.
	\end{aligned}
	\]
	It is a quick computation to see that $B$ and $C$ do not commute, and thus the map $\Delta^{\squarediv}\to Z$ we have constructed cannot be extended to a 3-simplex. As such, $Z$ is not 2-Segal. 

\end{proof}

%	To summarize, $Z$ is an example of a simplicial effect which is \emph{not} an effect algebra or effect algebroid. 

%inversion dicussion moved up

\subsection{{Cyclic cohomology and states}}

In this section we extend the observations of Roumen in \cite{RouCoh} {on the cohomology of effect algebras} to simplicial effects. 
The notion of simplicial effects introduced in {Definition \ref{def:simplicial effect}} extends effect algebroids, and in particular, effect algebras. 
We begin by extending  
{the notion of a state (Definition \ref{def:state})}
to cyclic sets.
Recall that the nerve functor $N:\Eff \to \SimpEff$ is fully faithful. The nerve $N[0,1]$ will play the role of the unit interval.

\begin{definition}\label{def:state cyclic set}
A cyclic set morphism $\varphi:X \to {N[0,1]}$ is called a \emph{state} {on $X$}. We will write $\St(X)$ to denote the set of states {on $X$}.
\end{definition}

Unraveling this definition we find that $\St(X)$ consists of functions 
%\[
%\varphi:X_1\to [0,1]
%\]
\[
\begin{tikzcd}
\varphi:X_1 \arrow[r] & {[0,1]}  
\end{tikzcd}
\]

such that 
\begin{itemize}
\item  $\varphi(\tau_1x) = 1-\varphi(x)$,
\item  $\varphi(d_1\sigma) = \varphi(d_2\sigma) +\varphi(d_0\sigma)$ for all $\sigma\in X_2$.
\end{itemize}
For $x\in X_0$ let us write $0_x=s_0(x)$ and $1_x=\tau_1(0_x)$.
The second condition implies $\varphi(0_x)=0$ and $\varphi(1_x)=1$.
{In particular, for an effect algebra $E$ we have $\St(E)=\St(NE)$.}

We will use Connes' definition of cyclic cohomology since we are interested in $\RR$ as the coefficient ring. 
Given a cyclic set $X$ we can define a cyclic $\RR$-module $\RR[X]$ where $\RR[X]_n$ is the free $\RR$-vector space generated on $X_n$.
There is an associated chain complex $C_\bullet(X)$ defined in the usual way using the underlying simplicial structure.
The cyclic action in each degree induces an action on the chain complex.
Dually there is a cochain complex $C^\bullet(X)$ which is equipped with a cyclic action.
Following \cite{Loday}, we also consider the sign correction in this action.
More explicitly, we can identify elements of $C^n(X)$ with functions $f:X_n\to \RR$. The cyclic action is defined by
$$
t_n\cdot f(x) = (-1)^{n} f(\tau(x)).
$$
The cyclic cohomology of $X$ is defined as the cohomology of the cochain complex $C^\bullet_\lambda(X)$ which, in each degree, consists of the cochains invariant under the cyclic action, i.e.,
\[
C^n_\lambda(X) = \set{f:X_n\to \RR:\, t_n\cdot f =f }.
\]  
The $n$-th cyclic cohomology of $X$ is denoted by $\HC^n(X)$.
We are interested in the first cyclic cohomology group. Unraveling the definition we find that it consists of functions $f:X_1\to \RR$ such that
\begin{itemize}
\item $f(\tau_1x) = -f(x)$,
\item $f(d_1\sigma) = f(d_2\sigma) +f(d_0\sigma)$ for all $\sigma\in X_2$.
\end{itemize} 
%We observe that $\HC^1(X)$ is closely related to $\St(X)$. 

\begin{theorem}
Let $X$ be a cyclic set satisfying the following two conditions:
\begin{enumerate}
\item $\St(X)$ is non-empty.
\item The difference map 
	\[
		\begin{tikzcd}[row sep=0em]
			\Set_\Delta(X,N\RR_{\geq 0})^{\times 2}\arrow[r] & \Set_\Delta(X,N\RR)  \\
			(f_1,f_2) \arrow[r,mapsto] & f_1-f_2
		\end{tikzcd}
		\]
is surjective.
\end{enumerate}
Then $\HC^1(X)$ is the smallest vector space containing $\St(X)$.
\end{theorem}
\begin{proof}
The argument is similar to \cite[Theorem 4.3]{RouCoh}. For a state $\varphi_0$ on $X$ we define an affine map $i:\St(X)\to \HC^1(X)$ by $\varphi\mapsto \varphi-\varphi_0$. Then given an affine injection $j:\St(X)\to V$, where $V$ is an arbitrary $\RR$-vector space, we define $k: \HC^1(X)\to V$ by the following formula: Given $g\in \HC^1(X)$, first express $g+\varphi_0$ as a difference $g_1-g_2$ where $g_i:X\to N\RR_{\geq 0}$ are simplicial set maps. If $g_i(1)\neq 0$ then define $\tilde g_i(x) = g_i(x)/g_i(1)$. Finally, set
\[
{k(g)} = \left\lbrace
\begin{array}{ll}
 g_1(1)j(\tilde g_1) - g_2(1) j(\tilde g_2)  & g_i(1)\neq 0,\; i=1,2 \\
 - g_2(1) j(\tilde g_2)  & g_1(1)= 0,\; g_2(1)\neq 0\\
 g_1(1) j(\tilde g_1)  & g_1(1)\neq 0,\; g_2(1)= 0\\  
 0 & \text{otherwise.}
\end{array}
\right.
\]
Showing that definition of $k$ does not depend on the difference decomposition and that $j=k\circ i$ follows \cite[Theorem 4.3]{RouCoh}.
%\comm{It seems to me that Roumen's proof generalizes to this case. However, his statement does not involve effect algebroids, but instead effect algebras. Do we need reducedness here?} 
\end{proof}

\begin{example}\label{ex:states}
{We now present two examples illustrating contrasting behaviors of states and first cyclic cohomology in different cyclic sets.}
\begin{enumerate}
\item Let us consider the cyclic set $P_\hH(N\ZZ/{_d})$. We have that
\[
\St(P_\hH(N\ZZ/{_d}))=\emptyset \;\;\; \text{ and }\;\;\; \HC^1(P_\hH(N\ZZ/{_d}))=0. 
\]
This can be observed by identifying this cyclic set with \( N(\mathbb{Z}_{/d}, U(\mathcal{H})) \) using Proposition \ref{prop:U(H)_P(H)} and noting that any state represented by a function \( f: U(\mathcal{H}) \to [0,1] \) satisfies \( f(A^d) = f(\one) = 0 \). This implies that \( f \) multiplied by \( d \) is zero.

\item Our next example supports non-trivial states. {For details see \cite{sdist}. Let $S^1$ be the simplicial circle.}
For $\dim(\hH)\geq 3$ we have
\[
\St(P_\hH(S^1))=\Den(\hH) \;\;\; \text{ and }\;\;\; \HC^1(P\hH(S^1))=\Herm(\hH)
\]
where $\Den(\hH)$ and $\Herm(\hH)$ denote the density operators and Hermitian operators, respectively. In details, Proposition \ref{pro:ES1} gives $P_\hH(S^1)\cong N(\Proj(\hH))$. Therefore $\St(P_\hH(S^1))$ can be identified with the states on the effect algebra of projectors. 
Then Theorem \ref{thm:Gleason} gives a bijection
	\[
		\begin{tikzcd}[row sep=0em]
		\Den(\hH) \arrow[r] & \St( P_\hH S^1 ) \\
			\rho \arrow[r,mapsto] & (\Pi \mapsto \Tr(\rho\Pi))
		\end{tikzcd}
		\]
A version of Gleason's theorem for Hermitian matrices gives a similar identification for the first cyclic cohomology; see \cite[Chapter 3]{dvurecenskij2013gleason}.
\end{enumerate}
\end{example}

%{The upshot is that   we recover the states of effect algebras without obtaining any new, interesting states from weakly associative partial groups.
%}

\subsection{{States on simplicial effects}}

{Next, we show that our key example (Construction \ref{const:key example}) supports non-trivial states.}
Recall that $Z$ is the full simplicial subset of $P_\hH N\ZZ_{/3}$ on the $2$-truncation $X$ whose $2$-simplices consist of projective measurements $\Pi=\set{\Pi^{ab}}_{a,b}$  
satisfying 
\begin{equation}\label{eq:conditions}
\Pi^{11}=\Pi^{21}=\Pi^{12}=\zero.
\end{equation}

We will use the identification in Proposition \ref{prop:U(H)_P(H)}.
A $2$-simplex can be described as a pair $(A,B)$ of commuting unitary matrices such that $A^3=B^3=\one$. Simultaneous diagonalization gives 
\begin{align*}
A &= \Pi_A^0 + \omega \Pi_A^1 + \omega^{2} \Pi_A^{2} \\
B &= \Pi_B^0 + \omega \Pi_B^1 + \omega^{2} \Pi_B^{2}. 
\end{align*}
Using $\Pi^{ab} = \Pi^a_A \Pi^b_B$ the additional condition in Equation \ref{eq:conditions} imply that 
\[
\Pi^1_B \perp \Pi^a_A\;\;\; \text{ and }\;\;\;
\Pi^1_A \perp \Pi^a_B
\] 
for $a=1,2$. Let us write $\Pi\leq \Pi'$ if the image of $\Pi$ is contained in the image of $\Pi'$. The first orthogonality relations imply that  $\Pi_B^1\leq \Pi_A^0$ and $\Pi_A^{2}\leq \Pi_B^{2}$. Similarly the second gives $\Pi_A^1\leq \Pi_B^0$ and $\Pi_B^{2}\leq \Pi_A^{2}$. Combining these two conditions we obtain
\begin{align*}
\Pi_B^0 &= \Pi_A^1 +\Pi \\
\Pi_B^1 &= \Pi_A^0 -\Pi\\
\Pi_B^{2} &=  \Pi_A^{2}
\end{align*}
for some $\Pi\leq \Pi_A^0$. Writing $C=AB$ we obtain 
\[
\Pi_C^0 = \Pi,\;\;\; \Pi_C^{1} = \bar \Pi,\;\;\; \Pi_C^{2} = \zero
\]
where $\bar \Pi$ stands for $\one-\Pi$. Let $\varphi:Z_2\to [0,1]$ be a state. Representing the  projectors in the spectral decomposition of the unitary operators as columns we have
\begin{equation}\label{eq:partial-additive}
\varphi\column{\Pi^0_A}{\Pi^1_A}{\Pi^{2}_A} + 
\varphi\column{\Pi^1_A+\Pi}{\Pi^0_A-\Pi}{\Pi^{2}_A}  = \varphi\column{\Pi}{\bar \Pi}{\zero}. 
\end{equation}

\begin{proposition}
Let $Z$ denote the simplicial effect defined in Construction \ref{const:key example}. The following map is a bijection:
	\[
		\begin{tikzcd}[row sep=0em]
		\Den(\hH) \arrow[r] & \St( Z ) \\
			\rho \arrow[r,mapsto] & \left(\Pi \mapsto \Tr(\rho(\bar \Pi^0 -\frac{1}{2} \Pi^2)\right)
		\end{tikzcd}
		\]
where $\Pi:\ZZ_{/3}\to \Proj(\hH)$ is a projective measurement.
\end{proposition}
\begin{proof}
We begin by observing that the function associated to $\rho$ satisfies Equation \ref{eq:partial-additive}, thus it is indeed a state on $Z$. It remains to show that all states on this simplicial effect are of this form. 

We begin by specializing Equation \ref{eq:partial-additive} to the extreme cases $\zero \leq \Pi \leq \Pi^0_A$.
\begin{itemize}
\item Setting $\Pi=\zero$ gives
\[
\varphi\column{\Pi^0_A}{\Pi^1_A}{\Pi^{2}_A} + 
\varphi\column{\Pi^1_A}{\Pi^0_A}{\Pi^{2}_A}  = 1. 
\]
Further setting $\Pi_A^2=\zero$ we obtain
\begin{equation}\label{eq:swap-orth}
\varphi\column{\Pi^0_A}{\bar \Pi^0_A}{\zero} + 
\varphi\column{\bar \Pi^0_A}{\Pi^0_A}{\zero}  = 1. 
\end{equation}
\item  Setting $\Pi=\Pi_A^0$, we obtain
\begin{equation}\label{eq:special-case}
\varphi\column{\Pi^0_A}{\Pi^1_A}{\Pi^{2}_A}  
 = \varphi\column{\Pi_A^0}{\bar \Pi_A^0}{\zero}
-
\varphi\column{\bar \Pi^{2}_A}{\zero}{\Pi^{2}_A}  
 .
\end{equation}
Further setting $\Pi_A^1=\zero$ in this equation we obtain 
\begin{equation}\label{eq:half}
2\varphi\column{\Pi^0_A}{\zero}{\Pi^{0}_A}   = \varphi\column{\Pi_A^0}{\bar \Pi_A^0}{\zero}.
\end{equation}
In particular, setting $\Pi_A^0=\zero$ this gives us $\varphi(\omega^{2}\one) = 1/2$.
\end{itemize}
Next, we apply Equation \ref{eq:special-case} to the two terms on the left hand side of Equation \ref{eq:partial-additive}. In addition, using Equation \ref{eq:half} we obtain
\begin{equation}\label{eq:third-zero}
\varphi\column{\Pi_A^0}{\bar \Pi_A^0}{\zero} +
\varphi\column{\Pi_A^1+\Pi}{\overline {\Pi_A^1+\Pi}}{\zero} =
\varphi\column{\bar \Pi_A^{2}}{ \Pi_A^{2}}{\zero} + \varphi\column{\Pi}{\bar \Pi}{\zero}.
\end{equation} 
Next, setting $\Pi=\zero$ and using Equation \ref{eq:swap-orth} gives
\[
\varphi\column{\bar \Pi^0_A}{\Pi^0_A}{\zero} +
\varphi\column{\bar \Pi^1_A}{\Pi^1_A}{\zero} =
\varphi\column{\Pi^2_A}{\bar\Pi^2_A}{\zero} .
\] 
This implies that restricting to the second factor induces a function
\[
\begin{tikzcd}[row sep=0em]
		\St(Z) \arrow[r] & \St( P_\hH(N\ZZ_{/3}) ) \\
		\varphi \arrow[r,mapsto] & (\Pi \mapsto \varphi(\Pi^1) )
		\end{tikzcd}
\]
By part (2) of Example \ref{ex:states}, this proves the surjectivity of the map in the statement of the Proposition.
\end{proof}

  \bibliography{bib.bib}
\bibliographystyle{ieeetr}

\end{document}